\documentclass[12pt,leqno]{article}
\usepackage{amssymb,amsmath,amstext}
\usepackage{amsthm}
\usepackage{subcaption}
\usepackage{authblk}
\usepackage{xcolor}
\usepackage{float}

\newtheorem{theorem}{Theorem}[section]
\newtheorem{lemma}[theorem]{Lemma}

\newtheorem{remark}[theorem]{Remark}
\newtheorem{example}[theorem]{Example}
\newtheorem{corollary}[theorem]{Corollary}

\DeclareMathOperator\arctanh{arctanh}
\usepackage[pdftex]{graphicx}
\newcommand{\drawat}[3]{\makebox[0pt][l]{\raisebox{#2}{\hspace*{#1}#3}}}

\newcommand{\rz}{I \hskip - 4.0pt R^N}

\def\R{\mathbb{R}}

\def\L{\mathcal{L}}
\def\K{\mathcal{K}}

\def\rz{\mathbb{R}}

\newcommand{\tx}[1]{\mbox{\;{#1}\;}}
\numberwithin{equation}{section}
\begin{document}
\title{On the variety of solutions of 1-dimensional nonlinear eigenvalue problems}
\author[1]{Catherine Bandle}
\author[2]{Simon Stingelin}
\author[3]{Alfred Wagner}
\affil[1]{Department of Mathematics, University of Basel, Switzerland}
\affil[2]{Institute of Applied Mathematics and Physics (IAMP), ZHAW School of Engineering, Switzerland}
\affil[3]{Department of Mathematics, RWTH Aachen, Germany}
\pagestyle{myheadings}
\markright{\sc }
\maketitle
\large
\bigskip

\abstract{Second order nonlinear eigenvalue problems are considered for which
the spectrum is  an interval. The boundary conditions
are of Robin and Dirichlet type. The shape and the number
of solutions are discussed by means of a phase plane
analysis. A new type of asymmetric solutions are discovered. Some
numerical illustrations are given.}
\bigskip

{\bf  Key words}: Nonlinear eigenvalue problems, Robin boundary conditions, symmetric and asymmetric solutions, phase plane analysis.
\bigskip

\section{Introduction}
In this paper we study one-dimensional boundary value problems of the following type
\begin{align}\label{general}
u''(x) + \lambda f(u)=0 \tx{in} (0,L),\quad   \lambda >0
\end{align}
under the boundary conditions
\begin{align}\label{Robin}
u'(0)= \alpha u(0), \quad u'(L)=-\alpha u(L), \: \alpha \in \rz, \: \alpha \ne 0.
\end{align}
The nonlinearity satisfies
\begin{align} \label{conditions}
f(u)>0, \: f'(u)\geq 0 \tx{for all} u\in \rz \tx{and} \lim_{u\to \infty} \frac{f(u)}{u}=\infty.
\end{align}
The higher dimensional version  $\Delta u + \lambda f(u)=0$ in $\Omega \in \R$ , $\frac{\partial u}{\partial \nu}+ \alpha u=0$ on $\partial \Omega$, $\nu$ outer normal and $\alpha >0$
arises in various models, such as combustion, thermal explosions and gravitational equilibrium of polytrop stars. Among the first problems of interest was the celebrated Gelfand problem \cite{Ge} where $f(u)= e^u$.  
The mathematical analysis was carried out in a series of papers see for instance \cite{KeCo}, \cite{HeaWa}, \cite{KeenerKe} and the references cited therein. 

It turns out
that there is a $0<\lambda^*<\infty$ such that the problem is solvable for $\lambda \in (0,\lambda^*)$ and no solution exists if $\lambda >\lambda^*$. The analysis was carried out in \cite{CrRa} and in \cite{Amann} where it was proved that for $\lambda<\lambda^*$ there exist at least two solutions. One of them is the minimal solution which is smaller than any other solution.

The radial solutions in balls have been studied in \cite{JoLu} for Dirichlet and in \cite{BeGaPi} for Robin boundary conditions with $\alpha>0$. Surprisingly the number of solutions depends on the dimension.

The 1-dimensional case with Dirichlet boundary conditions was first considered by Bratu \cite{Br} in the context of integral  equations. 
Some specific properties of this problem are discussed in \cite{Fink+}.
\bigskip

\noindent
So far little attention has  been  paid to  negative $\alpha$. To our knowledge only the eigenvalue problem $\Delta \phi + \lambda \phi=0$ in $\Omega$ with $\frac{\partial \phi}{\partial \nu} +\alpha \phi=0$ on $\partial \Omega$, $\nu$ outer normal, have been taken into consideration. It turns out that the behavior of the lowest eigenvalue differs significantly from the one with positive $\alpha$. For more details we refer to \cite{BaWa} and the references cited therein.
\medskip

\noindent
In this paper we study the various  solutions for both positive and negative $\alpha$. If $\alpha>0$, the solutions are positive and symmetric with repsect to $\frac{L}{2}$ whereas if $\alpha$ is negative,  the solutions have to change sign or are negative and asymmetric solutions appear.  To get as complete a picture as possible of the various solutions  we restrict ourselves to one-dimension. Many results hold in higher dimensions. 
\medskip

Our paper is organized as follows. We first collect some general properties of nonlinear eigenvalue problems. Section 3 is devoted to the phase plane analysis and to the main results. Section 4 contains some numerical results.

\section{Preliminaries}
By means of the Green's function problem \eqref{general} can be transformed into an integral equation. 
The Green's function $g(x,\xi)$ satisfies
$$
g_{xx}(x,\xi)=-\delta_\xi(x), \quad g_x(0,\xi)=\alpha g(0,\xi), \quad g_x(L,\xi)=-\alpha g(L,\xi).
$$
It is of the form
\begin{align}\label{Green}
g(x,\xi)=
\begin{cases}
\frac{1}{\alpha(2+\alpha L)}(1+\alpha L-\alpha \xi)(1+\alpha x)&\tx{if} x<\xi,\\
\frac{1}{\alpha(2+\alpha L)}(1+\alpha \xi)(1+\alpha L-\alpha x)& \tx{if} x\geq \xi.
\end{cases}
\end{align}
For $\alpha >0$ it is strictly positive and exists for all $\alpha$ whereas for $\alpha<0$ it  exists only if $2+\alpha L\neq 0$ and $\alpha \ne 0$. For $\alpha <0$ it may change sign.
The solution of \eqref{general} can be written as
\begin{eqnarray}\label{Greenb}
u(x)&=&\frac{\lambda}{\alpha(2+\alpha L)}\int_0^x(1+\alpha \xi)(1+\alpha L-\alpha x)f(u(\xi))\:d\xi\\
\nonumber&& +\frac{\lambda}{\alpha(2+\alpha L)}\int_x^L(1+\alpha L-\alpha \xi)(1+\alpha x)f(u(\xi))\:d\xi.
\end{eqnarray}
This integral equation was the starting point for the existence proofs derived in \cite{Br}.

\begin{figure}[H]
	\centering
	\begin{subfigure}[t]{0.3\textwidth}
	\centering
	\includegraphics[width=50mm]{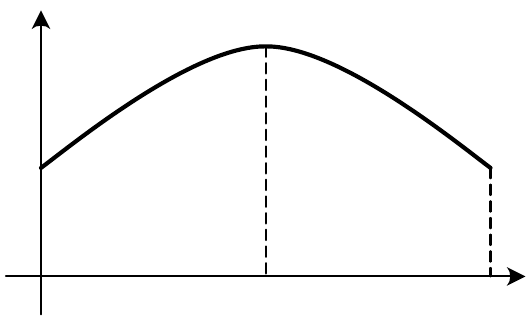}%
	\drawat{-28.mm}{1.5mm}{\scriptsize $L/2$}%
	\drawat{-5mm}{1.5mm}{\scriptsize $L$}%
	\drawat{-2mm}{6mm}{$x$}%
	\drawat{-49.8mm}{28mm}{$u$}%
	
	\caption{$\alpha > 0$, symmetric}
	\end{subfigure}\hspace{.5cm}%
	\begin{subfigure}[t]{0.3\textwidth}
	\centering
	\includegraphics[width=50mm]{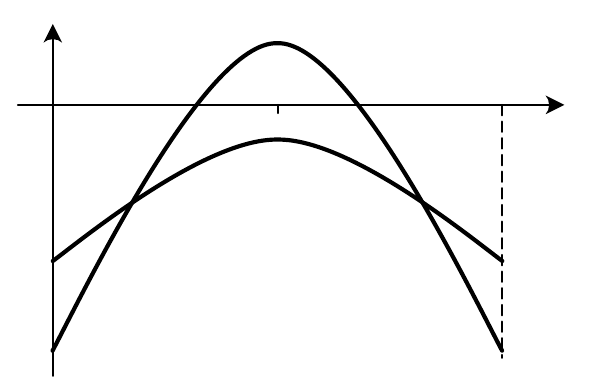}%
	\drawat{-29mm}{25mm}{\scriptsize $L/2$}%
	\drawat{-8mm}{25mm}{\scriptsize $L$}%
	\drawat{-3.5mm}{25mm}{$x$}%
	\drawat{-49mm}{29mm}{$u$}%
	
	\caption{$\alpha < 0$, symmetric}
	\end{subfigure}
	
	\begin{subfigure}[t]{0.3\textwidth}
	\centering
	\includegraphics[width=50mm]{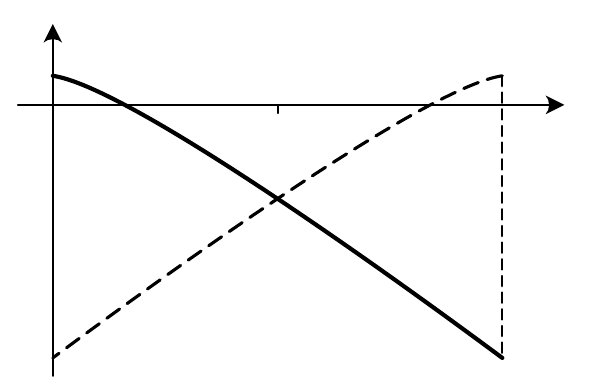}%
	\drawat{-29mm}{25mm}{\scriptsize $L/2$}%
	\drawat{-7mm}{25mm}{\scriptsize $L$}%
	\drawat{-3.5mm}{25mm}{$x$}%
	\drawat{-49mm}{29mm}{$u$}%
	
	\caption{$\alpha < 0$, decreasing, increasing}
	\end{subfigure}\hspace{.5cm}%
	\begin{subfigure}[t]{0.3\textwidth}
	\centering
	\includegraphics[width=50mm]{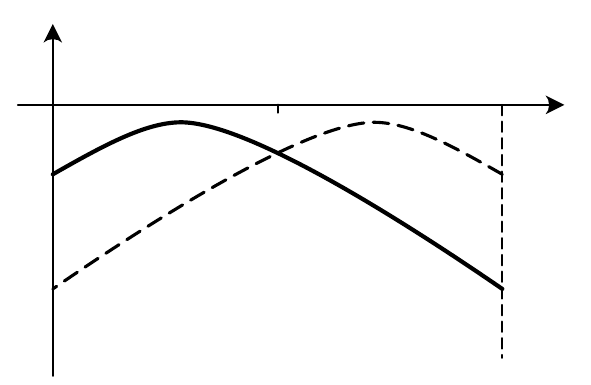}%
	\drawat{-29mm}{25mm}{\scriptsize $L/2$}%
	\drawat{-8mm}{25mm}{\scriptsize $L$}%
	\drawat{-3.5mm}{25mm}{$x$}%
	\drawat{-49mm}{29mm}{$u$}%
	
	\caption{$\alpha < 0$, non-monotone}
	\end{subfigure}%
	
	\caption{\label{fig:TypeOfSolutions}Type of solutions for $\alpha > 0$ and $\alpha < 0$.}
\end{figure}

\noindent
The following lemma follows immediately from the concavity of $u(x)$and from the conditions at the endpoints (see Figure \ref{fig:TypeOfSolutions}).

\begin{lemma} \label{elementary}
\begin{itemize}
\item[(i)] If $\alpha >0$,  
the solutions $u(x)$ of \eqref{general}, \eqref{Robin} are positive and symmetric.

\item[(ii)] If $\alpha<0$,three types of solutions can occur.
\begin{itemize}
\item[1.] $u(x)$ is monotone decreasing such that $u(0)>0$, $u'(0)<0$  and $u(L)<0$, $u'(L)<0$.
\item[2.] $u(x)$ is monotone increasing such that $u(0)<0$, $u'(0)>0$  and $u(L)>0,\:u'(L)>0$.
\item[3.] $u(x)$ is non-monotone such that $u(0)<0$, $u'(0)>0$ and $u(L)<0$, $u'(L)<0$. 
\end{itemize}
\end{itemize}
\end{lemma}
\begin{proof}
(i) Let $\alpha>0$. Since $u$ is concave, the inequality 
\begin{eqnarray}\label{ele1}
u'(L)\leq u'(x)\leq u'(0)
\end{eqnarray}
follows for all $0\leq x\leq L$. We assume that $u'(0)<0$ holds. Then $u'(0)=\alpha u(0)$ implies $u(0)<0$, 
and \eqref{ele1} implies $u'(L)<0$. The boundary condition in $x=L$ yields $u(L)>0$. On the other hand, 
$u'(0)<0$ and \eqref{ele1} imply $u'(x)<0$ for all $x\in[0,L]$. Thus $u$ is decreasing and hence $u(L)<u(0)<0$. 
This is a contradiction. Consequently $u'(0)>0$. The boundary condition implies $u(0)>0$. 
\medskip

\noindent
The boundary condition in $x=L$ implies an opposite sign of $u'(L)$ and $u(L)$. If $u'(L)>0$ then $u'>0$ on $[0,L]$. Then $u(0)>0$ implies $u(L)>0$ which contradicts the boundary condition. Hence $u'(L)<0$ and $u(L)>0$. The concavity then implies the positivity of $u$.
\medskip

\noindent
We will prove the symmetry of $u$ for $\alpha>0$. Since $u'(0)>0$, $u'(L)<0$ and $u$ is concave there exists a unique point $d\in (0,L)$ such that $u'(d)=0$. Note that both functions $u(x)$ and $u(2d-x)$ solve the differential equation $u''+\lambda\,f(u)=0$. Moreover
\begin{eqnarray*}
u(x)\vert_{x=d}=u(2d-x)\vert_{x=d}=u(d)\qquad\hbox{and}\qquad
u'(x)\vert_{x=d}= u'(2d-x)\vert_{x=d}=0.
\end{eqnarray*}
By the uniqueness property for solutions of initial value problem this implies $u(x)=u(2d-x)$ for all $x\in[0,d]$. Without loss of generality we can assume that $2d>L$. Let $x_0$ be chosen such that $2d-x_0=L$. Then $u'(L)=-\alpha\,u(L)$ is equivalent to $u'(x_0)=\alpha\,u(x_0)$. Since $u>0$ in $[0,L]$ we may consider the function
\begin{eqnarray*}
z(x):=\frac{u'(x)}{u(x)}. 
\end{eqnarray*}
Then $u''+\lambda\,f(u)=0$ implies
\begin{eqnarray*}
z'(x)+z^2(x)+\lambda\,\frac{f(u(x))}{u(x)}=0\quad\hbox{for}\quad x\in [0,L].
\end{eqnarray*}
Thus $z'(x)< 0$ on $(0,L]$. Since $z'(0)=\alpha$ and $z'(x_0)=\alpha$ we obtain $x_0=0$. Hence $d=\frac{L}{2}$. This implies the symmetry.
\medskip

\noindent
(ii) The boundary condition in $x=0$ implies: if
$u(0)>0$ then $u'(0)<0$. Since $u'(x)$ is decreasing the first assertion is immediate.
If $u(0)<0$ then $u'(0)>0$ and thus two possibilities can occur. Either $u'(L)>0$ or $u'(L)<0$. This proves the last claims.
\end{proof}
A consequence of Lemma \ref{elementary} (ii) 3.  is the possibility of the existence of symmetrical and asymmetrical solutions .
\medskip

\noindent
The spectrum of the linearized problem
\begin{align}\label{linearized_problem}
\phi'' + \mu f'(u) \phi=0 \tx{in} (0,L),\\
\nonumber  \phi'(0)=\alpha\, \phi(0), \quad \phi'(L)=-\alpha\, \phi(L),
\end{align}
is crucial for the stability of the solutions. If $f'(u(x))$ is a continuous bounded function
there exists a countable number of eigenvalues $\mu_1<\mu_2\leq \mu_3\hdots$. If $\alpha<0$, the lowest eigenvalue is negative.
\medskip

\noindent
We list some important - to a great extent well known - properties of nonlinear eigenvalue problems. 
\begin{lemma}\label{genprop1} Assume $\alpha >0$  and \eqref{conditions}.
\begin{enumerate}
\item If Problem \eqref{general}, \eqref{Robin} has a solution, there exists  a minimal solution which is smaller than any other solution. The minimal solution increases if $\lambda$ increases.\item For the  minimal solution  $0<\lambda \leq\mu_1$. 
\item If Problem \eqref{general}, \eqref{Robin} is solvable,  there  exists a number $0<\lambda^*<\infty$ such that it has a solution for any $0<\lambda \leq\lambda^*$. No solutions exist if $\lambda>\lambda^*$.
\item Under the additional assumption $f''(u)>0$, $ \mu_1\leq \lambda$ for non-minimal solutions. As a consequence non-minimal solutions intersect.
\end{enumerate}
\end{lemma}
\begin{proof}  The proof of the lemma is essentially due to Keller and Cohen \cite{KeCo} who proved these results for problems in higher dimensions.
\medskip

\noindent
1. Let $U(x)$ be any solution of \eqref{general}, \eqref{Robin}. By Lemma \ref{elementary} (1) it is positive. Consider the iteration process
\begin{eqnarray*}
u_0=0, \quad u_n'' + \lambda f(u_{n-1})=0, \quad  u_n'(0)=\alpha u_n(0), \quad u_n(L)=-\alpha u_n(L), \quad n=1,2\dots
\end{eqnarray*}
Since $u_0$ is a lower and $U(x)$ is an upper solution it follows that $u_0\leq u_n\leq u_{n+1} \leq U(x)$
The sequence $\{u_n \}_0^\infty$  is uniformly bounded and therefore $\lim_{n\to \infty} u_n=u(x)\leq U(x)$ where $u(x)$ is the minimal solution.
The minimal solution depends on $\lambda$. From its construction it follows that $u(x:\lambda_1)>u(x:\lambda_2)$ if $\lambda_1>\lambda_2$. 
\medskip

\noindent
2. The function $v(x):= \frac{\partial u(x:\lambda)}{\partial \lambda}$ is positive and satisfies
\begin{eqnarray*}
v_{xx} + \lambda f'(u)v + f(u)= 0 \tx{in} (0,L), \quad v_x(0)= \alpha v(0), \: v_x(L)=-\alpha v(L).
\end{eqnarray*}
By Barta's inequality \cite{PrWe}
\begin{eqnarray*}
\mu_1 \geq \min\left\{ -\frac{v_{xx}}{f'(u)v}\right \} > \lambda. 
\end{eqnarray*}
\medskip

\noindent
3. Since $f$ is superlinear and $f(0)>0$, there exits a positive number $\gamma$ such that $f(u)\geq \gamma u$ for all $u>0$. Hence 
\begin{eqnarray*}
0=u''(x) +\lambda f(u)\geq u'' +\lambda \gamma u.
\end{eqnarray*}
By Barta's inequality $\lambda \gamma \leq \kappa$,
where $\kappa>0$ is the lowest eigenvalue of 
\begin{eqnarray*}
\psi''+ \kappa \psi=0, \quad \psi'0)=\alpha \psi(0), \quad \psi'(L)=-\alpha \psi(L).
\end{eqnarray*}
4. Let $U$ be a non-minimal and $u$ the minimal solution. By the convexity of $f$
\begin{eqnarray*}
(U-u)'' + \lambda f'(U)(U-u)\geq 0,
\end{eqnarray*}
with boundary conditions
\begin{eqnarray*}
 (U-u)(0)=\alpha (U-u)(0)\quad \tx{and}\quad (U-u)(L)=-\alpha (U-u)(L).
\end{eqnarray*}
 Testing with $d:=U-u$ yields
 $$
 \lambda \geq \frac{\int_0^L (d')^2\:dx + \alpha (d^2(0)+d^2(L))}{\int_0^Lf'(U)d^2\:dx}.
 $$
 By the Rayleigh principle the expression at the right-hand side is bounded from below by $\mu_1$. Hence $\mu_1\leq \lambda$. Equality holds only if $d$ is the first eigenfunction. This is impossible by $f''(u)>0$.
\medskip

\noindent
 Suppose that $U_1\leq U_2$ are two non-minimal solutions which don't intersect. Then by the convexity of $f$
 $$
 (U_2-U_1)'' + \lambda f'(U_1)(U_2-U_1) \leq 0.
 $$
By Barta's inequality $\mu_1(U_1)\geq \lambda $. This is a contradiction to the previous result.
\end{proof}
Next we discuss problems  with negative $\alpha$. In contrast to positive $\alpha$ there appear also asymmetric solutions.
\begin{lemma} \label{alpha<0} Assume $\alpha<0$ and \eqref{conditions}.
Then
\begin{itemize}
\item[1.] For any solution the lowest eigenvalue of the linearized problem \eqref{linearized_problem} satisfies $\mu_1 <0<\lambda$.
\item[2.] Different solutions of Problem \eqref{general}, \eqref{Robin} intersect.
\item[3.] No minimal solution exists.
\item[4.] If $L=-\frac{\alpha}{2}$ the Green's function doesn't exists. Any solution of problem \eqref{general}, \eqref{Robin} with
$L =-\frac{\alpha}{2}$ must satisfy the compatibility condition
$$
\int_0^L (x-\frac{L}{2} )f(u(x))\:dx=0.
$$
\end{itemize}
\end{lemma}
\begin{proof}
\begin{itemize}
\item[1.]  Note that by  Rayleigh's principle
$$
\mu_1= \min_{\psi}\frac{\int_0^L \psi'^2\:dx + \alpha(\psi(0) +\psi(L))}{\int_0^Lf'(u)\psi^2 dx}.
$$
If we set $\psi =$ const., it follows that $\mu_1<0$.

\item[2.]  Suppose that there exist two solutions $U_2(x)\geq U_1(x)$. Then the difference $d=U_2-U_1$ is positive, concave
and satisfies the boundary conditions $d'(0)=\alpha d(0)<0$ and $d'(L)=-\alpha d(L)>0$.  This is impossible by the concavity of $d$. Hence $U_1$ and $U_2$ intersect. 
\item[3.] This is an immediate consequence of 2.
\item[4.]  The first  claim follows from \eqref{Green}.  Testing problem \eqref{general} with $\phi= x-\frac{L}{2}$ establishes the second claim.
\end{itemize}
\end{proof}
\section{Phase plane analysis}

Set
\begin{align*}
v(x)=\frac{u'(x)}{\sqrt{\lambda}} .
\end{align*}
Then the differential equation $u''+\lambda\,f(u)=0$ in \eqref{general} is transformed into a system of first order odes:
\begin{align}\label{P1}
v'=-\sqrt{\lambda} f(u), \quad u'=\sqrt{\lambda} v.
\end{align}
In a first step we analyze solutions of this system without taking the boundary conditions into account.
\medskip

\noindent
We define
\begin{align}\label{defF}
F(u)=
\begin{cases}
\int_{-\infty}^u\: f(t) dt & \tx{if $f$ is integrable at $-\infty$}, \\
\int_0 ^u\:f(t)dt & \tx{otherwise.}
\end{cases}
\end{align}
Then \eqref{P1} leads to the first order ode which can be integrated.
\begin{align}\label{phase1}
\frac{du}{dv} =- \frac{v}{f(u)} \Longleftrightarrow F(u)= C -\frac{v^2}{2}
\end{align}
and $F'(t)=f(t)$.
\medskip

\noindent
Since $f$ is positive, $F$ is strictly increasing and hence $F^{-1}$ exists. In view of \eqref{defF} we have:
$F^{-1}:\rz^+ \to \rz$ in the first case and $F^{-1}:\rz^+ \to \rz^+$ in the second case. As a consequence
\begin{eqnarray}\label{usol}
u=F^{-1}\left(C-\frac{v^2}{2}\right)\quad\hbox{for}\quad C>\frac{v^2}{2}.
\end{eqnarray}
From \eqref{defF} we also deduce
\begin{align*}
F(0)=
\begin{cases}
s_0:=\int_{-\infty}^0\: f(t) dt & \tx{if $f$ is integrable at $-\infty$}, \\
0 & \tx{otherwise.}
\end{cases}
\end{align*}
Hence $F^{-1}(s_0)=0$ in the first case and $F^{-1}(0)=0$ in the second case.
\begin{remark}\label{integrat}
Alternatively, equation \eqref{general} can be reduced to a first order differential equation, by multiplying $u''(x)+\lambda\,f(u(x))=0$ with $u'(x)$ and then by integrating
\begin{eqnarray}\label{P2}
\frac{du}{dx}=\sqrt{\lambda }v =\pm \sqrt{\lambda}\sqrt{2(C-F(u))}.
\end{eqnarray}
This ode can be integrated and gives an implicit formula for the solution $u$.
\end{remark}
\eqref{usol} implies that we can be represent the solutions in the $(v,u)$ plane as trajectories on the curves
$$
\mathcal{K}_C(v):=(v,u), \tx{where} u=F^{-1}(C -\frac{v^2}{2}),    \quad C>0.
$$
\begin{lemma}\label{KC}
Assume \eqref{conditions}. The function
\begin{eqnarray*}
v\to F^{-1}\left(C-\frac{v^2}{2}\right)
\end{eqnarray*}
is concave and symmetric with respect to the $u$ - axis. It is bounded from above and takes its maximum at $v=0$. Moreover the curve $\mathcal{K}_{C_1}(v)$ is below $\mathcal{K}_{C_2}(v)$  for $C_1<C_2$ and $\mathcal{K}_{C_1}\cap \mathcal{K}_{C_2} = \emptyset$
\end{lemma}
\begin{proof}
A straightforward calculation gives
\begin{eqnarray*}
\frac{d}{dv}F^{-1}\left(C-\frac{v^2}{2}\right)=-\,\frac{v}{F'\left(F^{-1}\left(C-\frac{v^2}{2}\right)\right)}
\end{eqnarray*}
and
\begin{eqnarray*}
\frac{d^2}{dv^2}F^{-1}\left(C-\frac{v^2}{2}\right)=
-
\,\frac{v^2\,F''\left(F^{-1}\left(C-\frac{v^2}{2}\right)\right)}{F^{'3}\left(F^{-1}\left(C-\frac{v^2}{2}\right)\right)}
-
\,\frac{1}{F'\left(F^{-1}\left(C-\frac{v^2}{2}\right)\right)}=
-\frac{v^2 f'(u)}{f^3(u)}-\frac{1}{f(u)}.
\end{eqnarray*}

Since $f>0$ and $f'\geq 0$ the concavity is shown.
\medskip

\noindent
Clearly the function $v\to F^{-1}\left(C-\frac{v^2}{2}\right)$ has a maximum in $v=0$ and is symmetric with respect to reflection $v\to -v$.

\noindent
The last statement in the lemma follows from the strict monotonicity of $F^{-1}$
\end{proof}
Figure \ref{fig:Figure1bMod} is typical for the class of nonlinearities considered in this paper and shows $\mathcal{K}_{C}$ for different values of $C$. 
\begin{figure}[H]
\centering
\includegraphics[width=80mm]{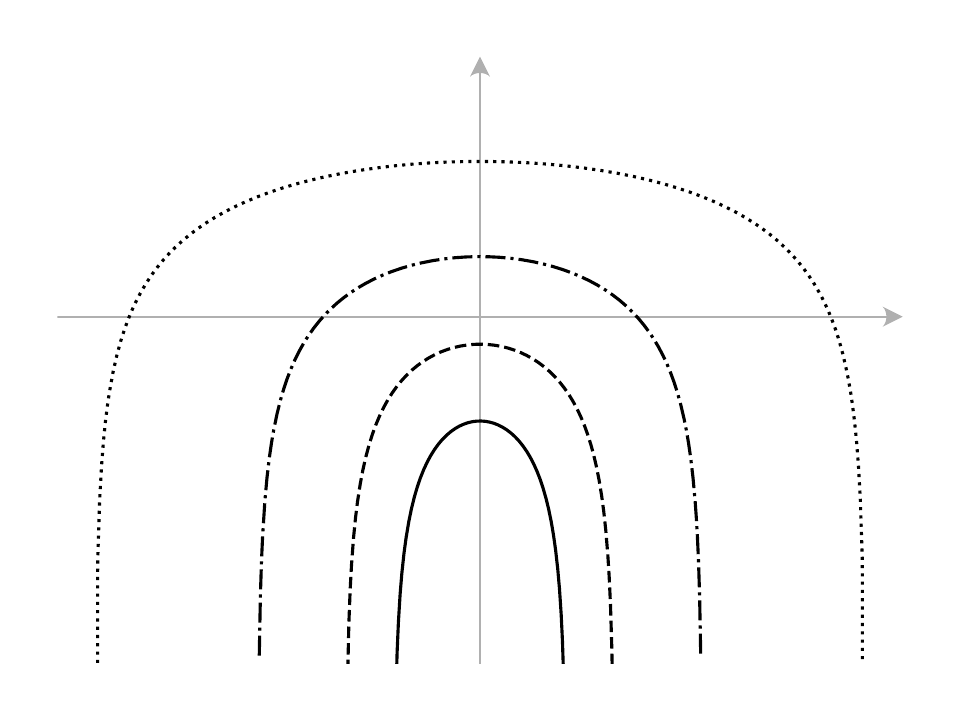}%
\drawat{-43.5mm}{53.5mm}{\small $u$}%
\drawat{-6.7mm}{35mm}{\small $v$}%

\caption{\label{fig:Figure1bMod}$\mathcal{K}_{C}$ for different values of $C$.}
\end{figure}


\noindent 
Next we introduce the boundary conditions in $x=0$ and $x=L$. Since $u'(0)=\alpha\,u(0)$ and $u'(L)=-\alpha\,u(L)$ we obtain
\begin{eqnarray}\label{bc}
v(0)=\frac{\alpha}{\sqrt{\lambda}}\,u(0)\quad\hbox{and}\quad
v(L)=-\frac{\alpha}{\sqrt{\lambda}}\,u(L)
\end{eqnarray}
Set
\begin{eqnarray*}
\mathcal{L}^+(v):=(v,\gamma^{*}v) \quad \tx{and}  \quad\mathcal{L}^-(v):=(v,-\gamma^{*} v) \quad \tx{where}  \quad \gamma^{*}:= \frac{\sqrt{\lambda}}{\alpha}.
\end{eqnarray*}
Then
\begin{eqnarray*}
(v(0),u(0))\in \mathcal{L}^{+} \quad\tx{and}\quad (v(L),u(L))\in \mathcal{L}^{-},
\end{eqnarray*}
see Figure \ref{fig:PhasePlane}.

Then another way to write \eqref{bc} is
\begin{align}\label{boundary}
u(0)=\gamma^{*}F^{-1}\left(C-\frac{1}{2}\left(\frac{u(0)}{\gamma^{*}}\right)^2\right), \quad  u(L)=-\gamma^{*}F^{-1}\left(C-\frac{1}{2}\left(\frac{u(L)}{\gamma^{*}}\right)^2\right).\\
\nonumber v(0)=\frac{1}{\gamma^{*}} F^{-1}\left(C-\frac{v^2(0)}{2}\right), \quad v(L)=-\frac{1}{\gamma^{*}} F^{-1}\left(C-\frac{v^2(L)}{2}\right).
\end{align}
\begin{figure}[H]
	\centering
	\begin{subfigure}{0.4\textwidth}\centering
	\includegraphics[width=80mm]{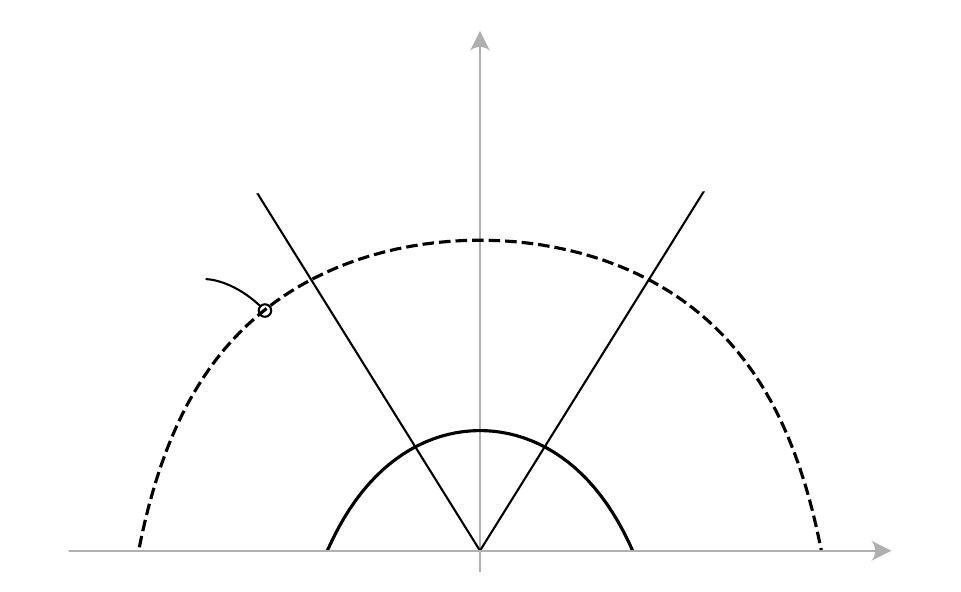}%
	\drawat{-60mm}{35mm}{$\mathcal{L}^-$}%
	\drawat{-22mm}{35mm}{$\mathcal{L}^+$}%
	\drawat{-74.mm}{27.5mm}{$\mathcal{K}_C(v)$}%
	\drawat{-44mm}{45mm}{$u$}%
	\drawat{-8mm}{5mm}{$v$}%
	
	\caption{$\alpha > 0$}
	\end{subfigure}%
	\hspace{1cm}
	\begin{subfigure}{0.4\textwidth}\centering
	\includegraphics[width=80mm]{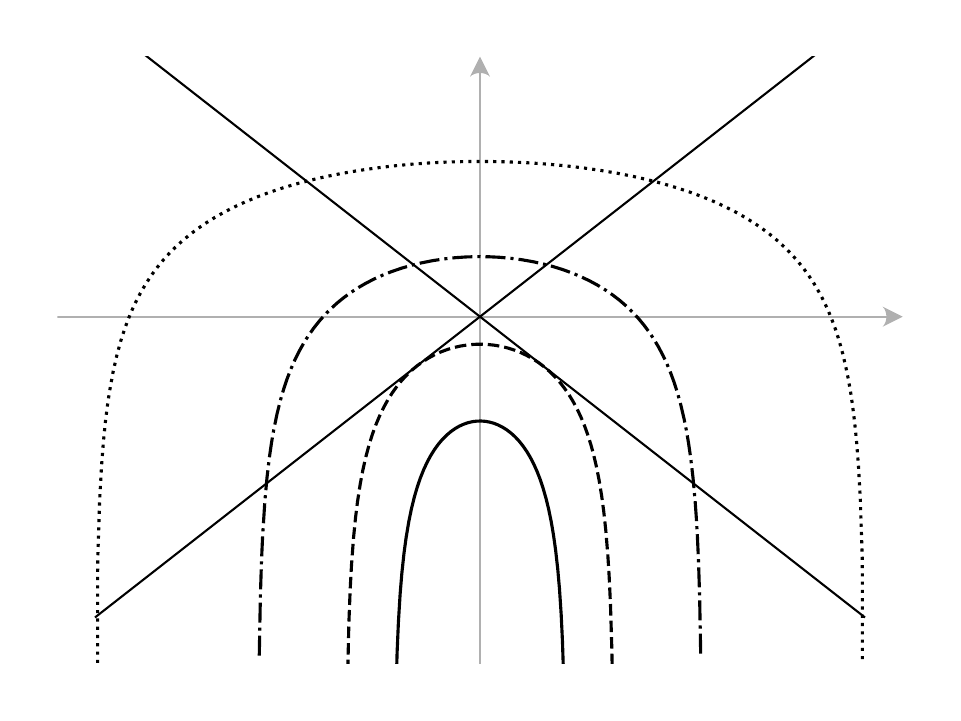}%
	\drawat{-60mm}{50mm}{$\mathcal{L}^+$}%
	\drawat{-23mm}{50mm}{$\mathcal{L}^-$}%
	\drawat{-44mm}{54mm}{$u$}%
	\drawat{-6mm}{30.5mm}{$v$}%
	
	\caption{$\alpha < 0$}
	\end{subfigure}
	
	\caption{\label{fig:PhasePlane}Phase plane.}
	
\end{figure}
In the case $\alpha>0$ Lemma \ref{genprop1} states that the solution $u$ is positive while for $\alpha<0$ either it changes sign or it is negative.
\medskip

\noindent
At this stage any intersection point of $\L^{\pm}$ with $\K_{C}$ corresponds to the solution which satisfies the 
Robin boundary condition in this intersection point. If $\alpha <0$, such points do not exist for all $C<\tilde C$. If they  exist, the trajectory between these intersection points corresponds to the solution to our problem. The intersection points of $\mathcal{L}^+$
 correspond to the left boundary values and intersection points of $\mathcal{L}^-$  correspond to the right boundary values. 
\medskip

\noindent
Let $P^{\pm}$ be the intersection points of $\L^{\pm}$ with $\K_{C}$.  The solution $\left(\frac{u'(x)}{\sqrt{\lambda}},u(x)\right)$ of \eqref{general} with $x\in[0,L=L(C)]$ corresponds to a trajectory on $\K_{C}$ between an intersection point $P^+$ and an intersection point $P^-$. This direction follows from  the concavity of $u(x)$. Indeed $v(x)$ decreases as $x$ increases. 
\medskip

\noindent
This part of the trajectory will be denoted with $\widehat{P^+P^-}$. The length $L(C)$ of the corresponding interval is then implicitly given by integration of \eqref{P1} and by \eqref{usol}. Since $v(0)>v(L(C))$
\begin{eqnarray}\label{length2}
L(C)=\frac{1}{\sqrt{\lambda}}\int_{v(L(C))}^{v(0)}\frac{dv}{f\left(F^{-1}\left(C-\frac{v^2}{2}\right)\right)}.
\end{eqnarray}
This length depends on $\alpha, \:\lambda$ and $C$. These
parameters are determined implicitly by $v(0)$ and $v(L)$ as follows, s. \eqref{boundary},
$$
C= F(\gamma^*v(0))+\frac{v(0)^2}{2}, \quad C=F(-\gamma^*v(L))+ \frac{v(L)^2}{2}, \quad \gamma^*=\frac{\sqrt{\lambda}}{\alpha}.
$$
\subsection{Trajectories in the phase plane for the Dirichlet problem}
In this subsection we discuss problem \eqref{general} with $u(0)=u(L)=0$ in the phase plane. Recall that its solutions if they exist, are positive. We set $ F(u)=\int_0^uf(s)ds$. Consequently $s_0=0$.The solutions are represented by the trajectories on $\mathcal{K}_C$ starting at $(\sqrt{2C}, 0)$ and ending at at $(-\sqrt{2C}, 0)$. The length is given by \eqref{length2}. In view of the symmetry of $u(x)$ with respect to $L/2$ it can be written in the form
$$
\frac{\sqrt{\lambda}}{2}L(C)= \int_0^{\sqrt{2C}}\frac{dv}{f(F^{-1}(C-\frac{v^2}{2}))}.
$$
The change of variable $t\sqrt{2C}=v$ leads to
\begin{align}\label{lengthDirichlet}
\frac{\sqrt{\lambda}}{2}L(C)= \int_0^1\frac{\sqrt{2C}\:dt}{f(F^{-1}(C(1-t^2)))}.
\end{align}
Let us now differentiate this expression with respect to $C$. Then, keeping in mind $u=F^{-1}(C(1-t^2))$, we obtain
\begin{align*}
\frac{d}{dC}\left ( \frac{\sqrt{\lambda}}{2} L(C)\right)&= \int_0^1\left\{ \frac{1}{\sqrt{2C}f(u)}-\frac{\sqrt{2C} f'(u)}{f^3(u)}(1-t^2)\right\}\,dt\\
&= \int_0^1\frac{1}{\sqrt{2C}f^3(u)}\{ f^2(u)-2f'(u)F((u)\}\,dt.
\end{align*}
The sign of $dL(C)/dC$ depends on $g(u):= f^2(u)-2f'(u)F(u)$, $u>0$. We have $g(0)= f^2(0)>0$ and $g''(u)=-2f''(u)F(u)$. Consequently
\begin{theorem} \label{g(u)}
(i) If $g(u)>0$ in $(0,\infty)$, $L(C)$ is monotone. In this case \eqref{general} with Dirichlet boundary conditions has at most one solution.

(ii) If $f''(u)>0$, \eqref{general} with Dirichlet boundary conditions has at most two solutions. The solutions are ordered.
\end{theorem}
\begin{example} 
Let $f(u)=e^u$. Then $F(u)= e^u-1$ and
$$
g(u)= e^u(2- e^u), \quad \frac{\sqrt{\lambda}}{2}L(C)= \int_0^{\sqrt{2C}}\frac{dv}{C+1-\frac{v^2}{2}}=\sqrt{\frac{2}{C+1}}\arctanh\sqrt{\frac{C}{C+1}}.
$$
The concavity of  $g(u)$ and $g(0)=1$ imply that $L(C)$ has only one critical point. Moreover $L(0)=L(\infty)=0$.
Recall that $u_{\max}=\ln(C+1)$.
\medskip

\noindent
Numerical computations indicate that there are non convex nonlinearities for which more solutions exist (see Section 4).
\end{example}
\subsection{Trajectories  in the phase plane for $\alpha>0$}
By Lemma \ref{elementary} only positive symmetric solutions of problem \eqref{general}.  Since $u>0$ for $\alpha>0$ we can choose, s. \eqref{defF} 
\begin{eqnarray*}
F(u)=\int_{0}^{u}f(t)\:dt,
\end{eqnarray*}
and $F^{-1}:\rz^+ \to \rz^+$ with $F^{-1}(0)=0$. Then $C\in (0, \infty)$ and $\mathcal{K}_0= (0,0)$. 
\medskip

\noindent
Let  $P^+=(v_1,u_1)$, $v_1,u_1>0$   be the intersection point of $\mathcal{L}^+$ with $\mathcal{K}_C$. Accordingly $P^-=(-v_1,u_1)$ is the intersection point of $\mathcal{L}^-$ with $\mathcal{K}_C$. 
 \medskip

\noindent
In the phase plane the  solutions of \eqref{general}  are represented by a trajectory on $\mathcal{K}_C(v)$
starting at $P^+=(v_1,u_1)\in \mathcal{L}+$ and ending at $P^-=(-v_1,u_1)\in \mathcal{L}^-$. Note that $v_1$ and $u_1$ depend (smoothly) on $C$.
\medskip

\noindent
Since $u$ is symmetric and smooth we have
\begin{eqnarray*}
v\left(\frac{L}{2}\right)=\frac{1}{\sqrt{\lambda}}u'\left(\frac{L}{2}\right)=0
\end{eqnarray*}
Hence by the symmetry and \eqref{length2}
\begin{eqnarray}\label{lambda1}
L=L(C)=2\int_{0}^{v(0)}\frac{dv}{\sqrt{\lambda}f\left(F^{-1}\left(C-\frac{v^2}{2}\right)\right)}.
\end{eqnarray}

\noindent
In the sequel we replace $v(0)$ by $v_1$ and determine $L(C)$ at $C=0$ and $C=\infty$. Clearly
 \begin{align} \label{C1}
L(0)=\lim_{C\to 0}2 \int _0^{v_1}\frac{dv}{\sqrt{\lambda}f\left(F^{-1}\left(C-\frac{v^2}{2}\right)\right)}=0.
\end{align}
If $v\in (0,v_1)$, then $u\geq \gamma^{*}v_1$. Since $f(u)$ is monotone increasing
\begin{eqnarray*}
\int _0^{v_1}\frac{dv}{\sqrt{\lambda}f(F^{-1}(C-\frac{v^2}{2}))}= \int _0^{v_1}\frac{dv}{\sqrt{\lambda}f(u)}\leq \frac{v_1}{\sqrt{\lambda}f(\gamma^{*}v_1)}=\frac{\alpha}{\lambda}\frac{ \gamma^{*}v_1}{ f(\gamma^{*}v_1)}.
\end{eqnarray*}
Differentiation of the boundary condition \eqref{boundary}, $\gamma^{*}v_1=u_1=F^{-1}(C-\frac{v_1^2}{2})$, or equivalently $F(\gamma^{*}v_1)+\frac{v_1^2}{2}=C$ with respect to $C$, implies
$v_1'(C)(\gamma^*f(\gamma^{*}v_1) +v_1)=1$. Thus $v_1(C)$ is monotone increasing.  From $
F(\gamma^{*}v_1)+\frac{v_1^2}{2}=C$ it follows that $v_1 \to \infty$ as $C\to \infty$.
\medskip

\noindent
The superlinearity  of $f$, s. \eqref{conditions}, implies
\begin{align}\label{C2}
L(\infty)=\lim_{C\to \infty} 2\int _0^{v_1}\frac{dv}{\sqrt{\lambda}f(F^{-1}(C-\frac{v^2}{2}))}=0.
\end{align}
We are now in position to establish the following result.
\begin{theorem}\label{Theo1}
Assume \eqref{conditions} and $f''(u)>0$. Let $\alpha$ and $\lambda$ be given positive numbers.  Then
$L(C)$ is bounded in $\mathbb{R}^+$. It satisfies  $L(0)=0$ and $L(C)\to 0$ as $C\to \infty$ and has exactly one critical point in $\mathbb{R}^+$. 
\end{theorem}
\begin{proof}
 From \eqref{lambda1}, \eqref{C1} and \eqref{C2} it follows that $L(C)$ is bounded in $\mathbb{R}^+$ and must have at least one critical point. 
Suppose 
that $L(C)$ has more than one critical point. Then there exists $L_0$ such that the equation $L(C)=L_0$ has at least three solutions $C_1<C_2<C_3$. The corresponding  trajectories on $\mathcal{K}_{C_i}$, $i=1,2,3$ are solutions
of problem \eqref{general} with $L=L_0$. Clearly $u_1(x)\leq u_2(x)\leq u_3(x)$. By Lemma \ref{genprop1} 4. non-minimal solutions have to intersect. This contradicts our assumption.
\end{proof} 
Let $L_{\max}(\lambda,\alpha)=\max_{C\geq 0}L(C)$. An immediate consequence of Theorem \ref{Theo1} is
\begin{corollary}\label{C01}
Under the same assumptions as for Theorem \ref{Theo1}, Problem \eqref{general} has for fixed $L\in (0,L_{\max}(\lambda,\alpha))$ two solutions, for $L=L_{max}(\lambda,\alpha)$ one and for $L>L_{\max}(\lambda,\alpha)$ no solutions.
\end{corollary}
Related results have been obtained by Keller and Cohen \cite{KeCo} and Laetsch \cite{Lae} for more general elliptic operators  and  problems in higher dimensions.
Laetsch was able to show that no three ordered solutions can exist.
\medskip

\noindent
Recall that $L(C)=L(C;\lambda,\alpha)$. We observe that $L(C)$ decreases as $\lambda$ increases
\begin{align} \label{LMo}
L(C;\lambda_2, \alpha)<L(C;\lambda_1,\alpha)\quad \tx{if} \quad \lambda_1<\lambda_2.
\end{align}
This is an immediate consequence of \eqref{lambda1} and the fact that $v(0)$ decreases as $\lambda$ increases.
\medskip

\noindent
Next we fix $L=L_0$ for some $L_0$ and rewrite \eqref{lambda1} as
\begin{eqnarray*}
\sqrt{\lambda(C)}:=\left(2\int_{0}^{v(0)}\frac{dv}{L_0\,f\left(F^{-1}\left(C-\frac{v^2}{2}\right)\right)}\right)^2.
\end{eqnarray*}
We will discuss $\lambda(C)$ for fixed $L=L_0$ and $\alpha>0$.

\begin{figure}[H]
	\centering
	\includegraphics[height=6cm]{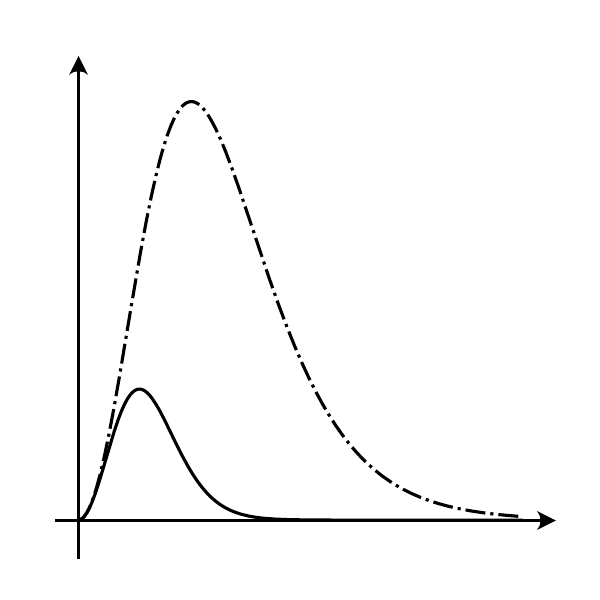}%
	\drawat{-6mm}{9.5mm}{$C$}%
	\drawat{-62.5mm}{52.5mm}{$\lambda(C)$}%
	
	\caption{$\lambda(C)$ for two different values of $\alpha$.}	
\end{figure}

\noindent
Since we will fix $\alpha$ we set 
\begin{eqnarray*}
L_{\max}(\lambda):=L_{max}(\lambda,\alpha)= \max_{C>0} L(C;\lambda,\alpha).
\end{eqnarray*}
Then the following holds true.
\medskip

\noindent
\begin{theorem} \label{curve1} 
Assume \eqref{conditions} and $f''(u)>0$. Let $\alpha>0$ be given. 
Then there exists $C^{*}>0$ such that $\lambda (C)$ is increasing in $(0,C^*)$ and decreasing 
in 
$(C^*,\infty)$. Moreover $\lambda(0)=0$, $\lambda(\infty)=0$ and $\lambda(C^*)=\lambda^*$.
\end{theorem}
\begin{proof}
Choose $\lambda$  such that $L_{\max}(\lambda)>L_0$. Then by  Theorem \ref{Theo1} the 
equation $L(C;\alpha,\lambda)=L_0$ has two solutions $C_1<C_2$. Hence Problem \eqref{general} has for $\lambda$ and $L=L_0$ two solutions $u(x)<U(x)$ such that $u_{\max}=F^{-1}(C_1)$ and $U_{\max}=F^{-1}(C_2)$. By \eqref{LMo}, 
  $L(C_i;\alpha, \lambda+\Delta \lambda)<L_0$ for $i=1,2$  and $\Delta \lambda >0$. Since $L(C;\lambda ,\alpha)$ is increasing in a neighborhood of $C_1$ and decreasing in a neighborhood of $C_2$ there exist $\Delta C_i>0$, i=1,2 such that
 $$ 
 L_0 = L(C_1+ \Delta C_1;\lambda + \Delta \lambda,\alpha) \tx{and} L_0 = L(C_2- \Delta C_2;\lambda + \Delta \lambda),\alpha).
 $$
Since for fixed $\lambda$, (s. Theorem \ref{Theo1}) $L(C)$ increases in $(0,L_{\max}(\lambda)$ and decreases in $(L_{\max}(\lambda)),\infty)$,  we have to increases $C_1$ in order to get a small solution $u(x)$ in $(0,L_0)$ for $\lambda +\Delta \lambda$. Similarly, in order to get a large solution $U(x)$ in $(0,L_0)$ for $\lambda+\Delta \lambda$ we have to decrease $C_2$.
\end{proof}
\medskip

This theorem together with Lemma \ref{genprop1} and \eqref{Greenb} leads to the following observation
\begin{lemma} \label{propertyUu}
Let $\alpha >0$ and $L<L_{\max}(\lambda)$ be given. Then problem \eqref{general} has for given $\lambda <\lambda^*$ two solutions
$u(x;\lambda)\leq U(x;\lambda)$. The minimal solution $u(x;\lambda)$ is a monotone increasing and $U(x;\lambda)$ is a monotone decreasing function of $\lambda$.
\end{lemma} 
\subsection{Trajectories in the phase plane for $\alpha<0$}
\subsubsection{General remarks}
As for $\alpha >0$ (s. Figure \ref{refa}) the solutions of \eqref{general} are given in the phase plane by trajectories on $\mathcal{K}_C$ such that $(v(0),u(0))\in \mathcal{L}^+$ and $(v(L),u(L)) \in \mathcal{L}^-$ .  Note that the slope of $ \mathcal{L}^+$ is  negative and the one of $ \mathcal{L}^-$ is positive (s. Figure \ref{refb}).  
\medskip

\noindent
According to \eqref{defF} , we define
\begin{eqnarray}\label{defF2}
F(u)=\int_{-\infty}^{u}f(t)\:dt.
\end{eqnarray}
Then $F^{-1}:\rz^{+}\to\rz$,  $F(0)=s_0=\int_{-\infty}^{0}f(t)\:dt$ and $F^{-1}(s_0)=0$.
\medskip

\noindent
We recall from \eqref{usol} that $u=F^{-1}\left(C-\frac{v^2}{2}\right)$ and $C$ is chosen such that $C>\frac{v^2}{2}$. From \eqref{P1} it follows that
\begin{eqnarray} \label{T-0}
\frac{du}{dv}=-\frac{v}{f\left(F^{-1}\left(C-\frac{v^2}{2}\right)\right)}\qquad \hbox{and}\qquad F'=f.
\end{eqnarray}
\begin{itemize}
\item From \eqref{defF2} and \eqref{usol} we deduce
\begin{eqnarray}\label{T-1}
v\,\to\,\pm\sqrt{2\,C}\quad\Longrightarrow\quad u(v)\,\to\,-\infty.
\end{eqnarray}
\item Since $f$ and $f'$ are integrable at $-\infty$, then
\begin{eqnarray*}
\lim_{s\to-\infty}f(s)=0.
\end{eqnarray*}
Hence
\begin{eqnarray*}
\lim_{v\,\to\,\pm\sqrt{2\,C}}f\left(F^{-1}\left(C-\frac{v^2}{2}\right)\right)=f(-\infty)=0,
\end{eqnarray*}
and by
\eqref{T-0}
\begin{eqnarray}\label{T-2}
v\,\to\,\pm\sqrt{2\,C}\quad\Longrightarrow\quad \frac{du}{dv}\,\to\,-\infty.
\end{eqnarray}
\item Since $0=F^{-1}(s_0)$, $v=\pm\sqrt{2(C-s_0)}$ implies that $u=0$. Then, for $C>s_0$
\begin{eqnarray}\label{T-3}
u(v)>0\quad\hbox{for}\quad v\in\left(-\sqrt{2(C-s_0)},\sqrt{2(C-s_0)}\right)\\
\nonumber  \tx{and} u(v)<0 \begin{cases}\tx{if} v>\sqrt{2(C-s_0)}\\
\tx{or} v<-\sqrt{2(C-s_0)}.
\end{cases}
\end{eqnarray}
If $C<s_0$, then $F^{-1}\left(C-\frac{v^2}{2}\right)<0$ for all $v$ with $C-\frac{v^2}{2}>0$. Hence
\begin{eqnarray}\label{T-4}
u(v)<0\quad\hbox{for}\quad v\in\left(-\sqrt{2C},\sqrt{2C}\right).
\end{eqnarray}
\end{itemize}
\medskip

\begin{figure}[H]
	\centering
	\begin{subfigure}{0.4\textwidth}\centering
	\includegraphics[width=80mm]{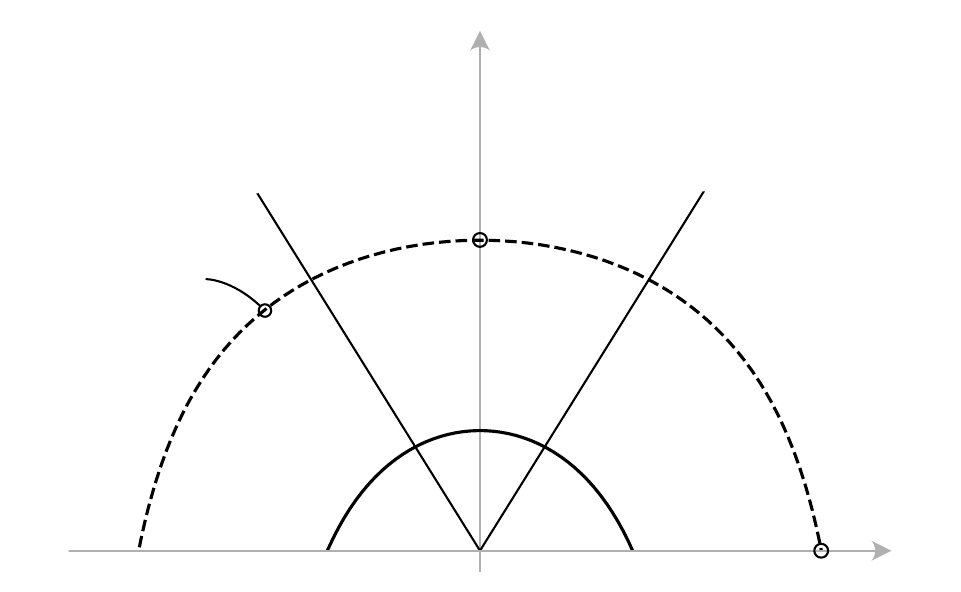}%
	\drawat{-60mm}{35mm}{$\mathcal{L}^-$}%
	\drawat{-22mm}{35mm}{$\mathcal{L}^+$}%
	\drawat{-40mm}{30.5mm}{$F^{-1}(C)$}%
	\drawat{-74.mm}{27.5mm}{$\mathcal{K}_C(v)$}%
	\drawat{-44mm}{45mm}{$u$}%
	\drawat{-8mm}{5mm}{$v$}%
	\drawat{-18mm}{-2mm}{$\sqrt{2C}$}%
	
	\caption{$\alpha > 0$}\label{refa}
	\end{subfigure}%
	\hspace{1cm}
	\begin{subfigure}{0.4\textwidth}\centering
	\includegraphics[width=80mm]{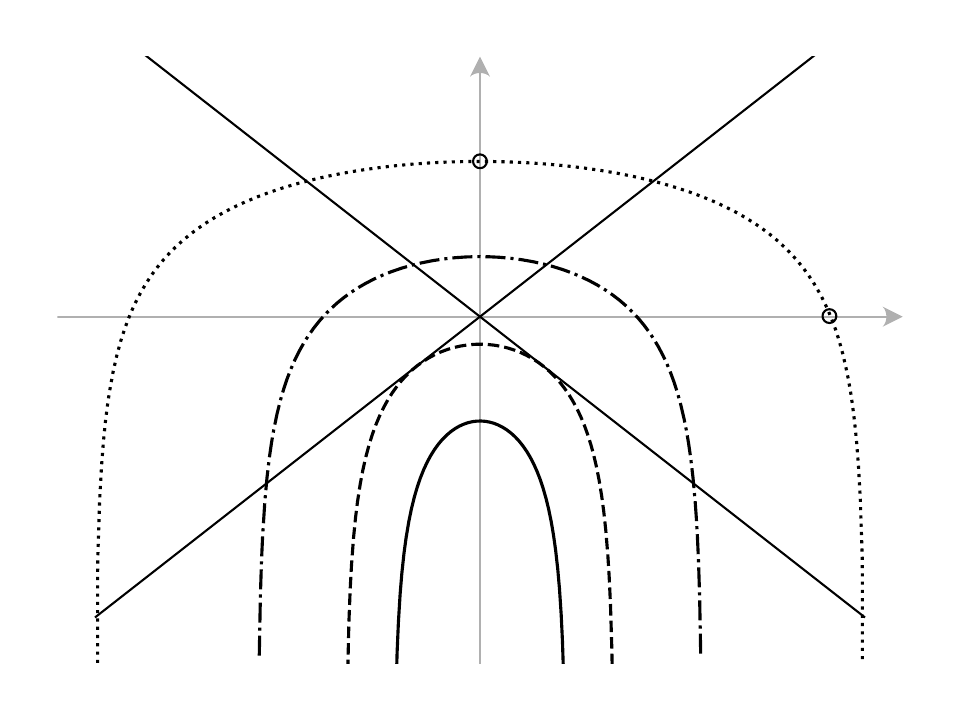}%
	\drawat{-60mm}{50mm}{$\mathcal{L}^+$}%
	\drawat{-23mm}{50mm}{$\mathcal{L}^-$}%
	\drawat{-40mm}{47.5mm}{$F^{-1}(C)$}%
	\drawat{-44mm}{54mm}{$u$}%
	\drawat{-6mm}{30.5mm}{$v$}%
	\drawat{-11mm}{35mm}{\small$\sqrt{2(C-s_0)}$}%
	
	\caption{$\alpha < 0$}\label{refb}
	\end{subfigure}
	
	\caption{Phase plane}
	
\end{figure}

\medskip

\noindent
\begin{lemma}\label{conta}
For any given $\gamma^{*}$ there exists a unique number $\tilde C<s_0$ such that $\mathcal{L}^\pm$ touch $\mathcal{K}_C$. Moreover for $C=\tilde C$ we have the following implicit characterizations: Set $P^{\pm}:=\mathcal{L}^{\pm}\cap\mathcal{K}_C$, then
\begin{itemize}
\item $P^{+}=(v_1,u_1)$ with $\gamma^{*} =-\,\frac{v_1}{f(F^{-1}(\tilde C-\frac{v_{1}^2}{2}))}$ and $u_1=F^{-1}\left(\tilde C-\frac{v_{1}^2}{2}\right)=\gamma^{*}v_1$;
\item $P^{-}=(-v_1,u_1)$ and $u_1=-\gamma^{*}(-v_1)$;
\item $\tilde{C}= F(\gamma^{*}v_1)+ \frac{v_1^2}{2}$.
\end{itemize}
\end{lemma}
\begin{proof} Since $\mathcal{K}_C$ is concave and since the family
$\mathcal{K}_C$ is ordered with respect to $C$, there exists a unique $\tilde C$ such that  $\mathcal{L}^+(v)$ touches $\mathcal{K}_{\tilde C}$ at the point $ P^+=(v_1,u_1)$ with $u_1=\gamma^{*}v_1$. 
\medskip

\noindent
As indicated in Figure \ref{refb} for such a point we have $v_1>0$ and in this point
\begin{eqnarray}\label{schnitt}
\gamma^{*}v_1=u_1=F^{-1}\left(\tilde C-\frac{v_{1}^2}{2}\right).
\end{eqnarray}
Since $\gamma^{*}<0$ for $\alpha<0$ this implies $F^{-1}\left(\tilde C-\frac{v_{1}^2}{2}\right)<0$. This in turn leads to $\tilde C<s_0$.
\medskip

\noindent
Similarly $\mathcal{L}^-$ touches  $\mathcal{K}_{\tilde C}$ at $P^-$ which is the reflexion of $\tilde P^+$ at the $u$-axis. Hence $P^-=(-v_1,u_1)$ and in this case $u_1=-\gamma^{*}(-v_1)$.
\medskip

\noindent
The tangents of $\mathcal{L}^+$ and $\mathcal{K}_{\tilde C}$ at the point of contact $P^+=(v_1,u_1)$ are the same and $P^+\in \mathcal{L}^+\cap \mathcal{K}_{\tilde C}$. Thus $v_1$ solves \eqref{schnitt} and
\begin{eqnarray*}
\gamma^{*} =-\,\frac{v_1}{f(F^{-1}(\tilde C-\frac{v_{1}^2}{2}))}=-\frac{v_1}{f(\gamma^{*}v_1)}.
\end{eqnarray*}
This implies
\begin{align} \label{contact}
\gamma^{*}f(\gamma^{*}v_1))+v_1=0\quad \tx{and}\quad \tilde C= F(\gamma^{*}v_1)+ \frac{ v_1^2}{2}.
\end{align}
\end{proof}
For $C>\tilde{C}$ there will be two intersections between $\mathcal{K}_C$ with $\mathcal{L}^+$.
These intersection points will be denoted by $P^{+}_{i} =(v^{+}_{i},u^{+}_{i})$ for $i=1,2$. The points are counted in such a way that  $v_1^+>v^+_2$ and $u_1^+<u_2^+$.
\medskip

\noindent
Analogously we set $P^-_i=(v^-_i,u^-_i)$, $i=1,2$ for the intersection points of $\mathcal{K}_C$ with $\mathcal{L}^-$ . The points are counted such that $v_1^-<v_2^-$  and  $u_1^+<u^-_2$. $P^+_1$ is the reflexion of $P^-_1$ and $P^+_2$  is the reflexion of $P^-_2$.
(see Figure \ref{fig:Figure3Mod}). 
\medskip

\begin{figure}[H]
	\centering
	\includegraphics[width=10cm]{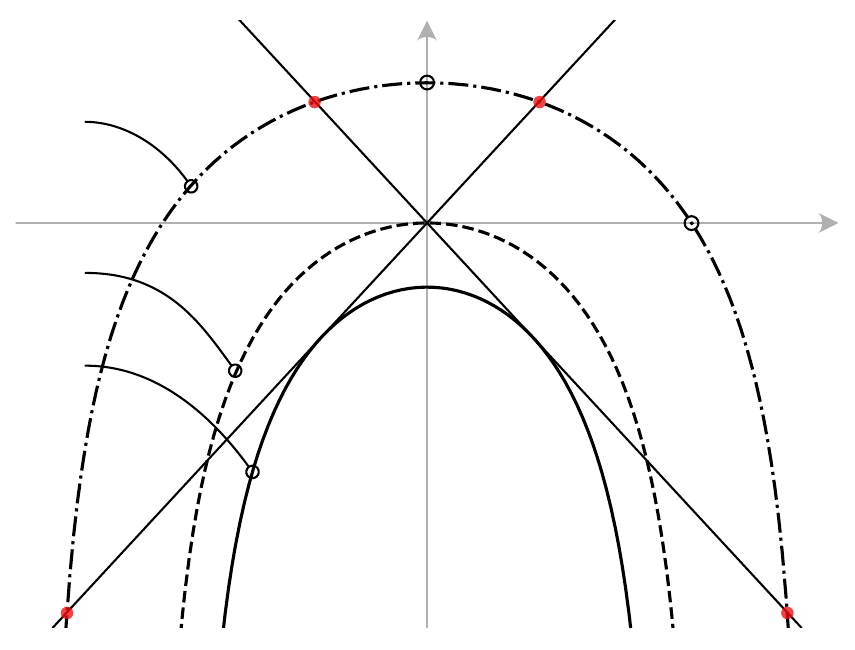}%
	\drawat{-75mm}{74.5mm}{$\mathcal{L}^+$}%
	\drawat{-27.5mm}{74.5mm}{$\mathcal{L}^-$}%
	\drawat{-49.5mm}{67.5mm}{\small$F^{-1}(C)$}%
	\drawat{-53.5mm}{72mm}{$u$}%
	\drawat{-4mm}{46.5mm}{$v$}%
	\drawat{-19.5mm}{51.5mm}{\small$\sqrt{2(C-s_0)}$}%
	\drawat{-96.5mm}{62mm}{$\mathcal{K}_C$}%
	\drawat{-96.5mm}{44mm}{$\mathcal{K}_{s_0}$}%
	\drawat{-96.5mm}{33mm}{$\mathcal{K}_{\tilde C}$}%
	\drawat{-98mm}{4.5mm}{$P^-_1$}%
	\drawat{-7mm}{4.5mm}{$P^+_1$}%
	\drawat{-39mm}{59.5mm}{$P^-_2$}%
	\drawat{-67mm}{59.5mm}{$P^+_2$}%
	
	\caption{\label{fig:Figure3Mod}Trajectories and intersection points}
\end{figure}
\medskip

The special form of $\mathcal{K}_C$ implies the following lemma.
\begin{lemma}
\begin{enumerate}
\item If $C<\tilde C$, $\mathcal{L}^\pm$ doesn't intersect $\mathcal{K}_C$.
\item If $\tilde C<C<s_0$, then $P^+_i$, $i=1,2$ are both on the right-hand side of the $u$-axis, whereas $P^-_i$ are on 
the left-hand side of the $u$-axis.
\item If $s_0<C$, $P^+_1$ is on the right-hand side and $P^+_2$ is on the left-hand side of the $u$-axis. Vice versa  $P^-_1$ is on the left-hand side and $P^-_2$ on the right-hand side of the $u$-axis.
 \end{enumerate}
 \end{lemma}
The trajectory $\widehat{P_i^+P_j^-}$ corresponds to a solution of \eqref{general}, \eqref{Robin} such that
 $L=L(C)$.
In contrast to positive $\alpha$ there exist beside of symmetric also asymmetric solutions.  All possibilities are listed  in Table \ref{tab:solutiontypes}, s. also Figure \ref{fig:Figure3Mod} .


\begin{table}[h]
\centering
\begin{tabular}{|l||r|r|r|}
\hline
\rule[12mm]{0mm}{-6mm}$s_0<C$ & $\widehat{P^+_1P^-_1}$ s-solution & $\widehat{P^+_1P^-_2}$ i-asolution & $\widehat{P^+_2P^-_1}$  d-solution \\ \hline
\rule[12mm]{0mm}{-6mm}$C=s_0$ & $\widehat{P^+_1P^-_1}$  s-solution & $\widehat{P^+_10}$ i-solution & $\widehat{0 P^-_1}$  d-solution \\ \hline
\rule[12mm]{0mm}{-6mm}$\tilde C<C<s_0$ &$\widehat{P^+_1P^-_1}, \widehat{P_2^-P_2^+}$ s-solutions& $\widehat{P^+_1P^-_2}$ c-solution& $\widehat{P^+_2P^-_1}$ c-solution \\ \hline
\rule[12mm]{0mm}{-6mm}$C=\tilde C$ & $\widehat{P^+_1P^-_1}=\widehat{P^+_2P^-_2}$  s-solution & no solution  & no solution \\ \hline
\rule[7mm]{0mm}{-2mm}$C<\tilde C$ & no solution & no solution & no solution  \\ \hline
\end{tabular}
\medskip

\parbox{145mm}{\caption{\label{tab:solutiontypes}s-solution = symmetric, i-solution = increasing asymmetric,\\d-solution = decreasing asymmetric, c-solution = non-monotone asymmetric}}
\end{table}

\subsection{Symmetric solutions}
The symmetric solutions are given by the trajectory $\widehat{P^+_iP^-_i}$, 
$i=1,2,$ in the phase plane. We denote by $L_{i}(C)$ the corresponding length of the interval.  By \eqref{length2} and \eqref{usol} we have
\begin{align} \label{LS}
L_{i}(C)=2\int_0^{v_i^+}\frac{dv}{\sqrt{\lambda}f(F^{-1}(C-\frac{v^2}{2}))}=2\int_{u_i^+}^{F^{-1}(C)}\frac{du}{\sqrt{2\lambda(C-F(u))}},
\end{align}
We start with the discussion of the trajectories $\widehat{P^+_2P^-_2 }$, thus $i=2$. They exist only if $C\in (\tilde C, s_0)$.
\medskip

\noindent

\begin{figure}[H]
\centering
\includegraphics[width=60mm]{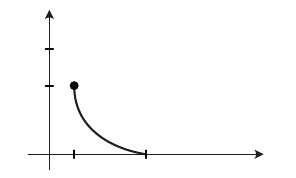}%
\drawat{-.9cm}{7mm}{$C$}%
\drawat{-31mm}{1mm}{$s_0$}%
\drawat{-47mm}{-0.5mm}{$\tilde{C}$}%
\drawat{-63mm}{17.5mm}{$L_2(\tilde{C})$}%
\drawat{-62.5mm}{25mm}{$-2/\alpha$}%
\drawat{-56mm}{33mm}{$L_2$}%

\caption{\label{fig:L2C}$L_2(C)$}
\end{figure}
\begin{theorem}\label{auxiliary1}
Let $\tilde C <C<s_0$ and let $\alpha<0$ and $\lambda>0$ be fixed. Then $L_{2}(\cdot)$ is a monotone decreasing function of $C$ such that $L_{2}(s_0)=0$ and $L_0:=L_{2}(\tilde C)<-\frac{2}{\alpha}$. For fixed  $L\in (0,L_0)$ there is exactly one trajectory $\widehat{P^+_2P^-_2}$ which corresponds to a  symmetric solution of  Problem \eqref{general} (see Figure \ref{fig:L2C}).
\end{theorem}
\begin{proof} From the phase plane it follows immediately  that $v_2^+(s_0)=0$ and that $v_2^+(C)$ decreases as $C$ increases. In addition $f(u)$ increases as $C\in (\tilde C, s_0)$ increases. Hence  $L_2(C)$ is monotone decreasing. From $v_2^+(s_0)=0$ and $f(0)>0$ we have
$L_{2}(s_0)=0$. 
\medskip

\noindent
Notice that  $\frac{d}{du} \sqrt {2(C-F(u))} =- \frac{f(u)}{\sqrt{2(C-F(u))}}$. Integration by parts of \eqref{LS} yields
 \begin{align} \label{length+}
 \frac{\sqrt{\lambda}}{2} L_{2}(C)= \frac{u_2^+}{\gamma^{*} f(u_2^+)} -\int_{u^+_2}^{F^{-1}(C)}\frac{f'(u)}{f^2(u)}\sqrt{2(C-F(u))}\:du.
 \end{align}
By \eqref{contact}, $u_2^+=-(\gamma^*)^2f(u_2^+)$.  Hence $L_{2}(\tilde C)< -\frac{2}{\alpha}$.
The last statement is obvious. 
\end{proof}
\noindent Differentiation of $L_{2}(C)$ with respect to $C$ yields
\begin{align*}
 \sqrt{\lambda}\frac{dL_{2}(C)}{dC}= 2\frac{(v_2^+)'}{f(u^+_2)} -2\int_0^{v_2^+}\frac{f'(u(v))}{f^3(u(v))}dv.
 \end{align*}
  From \eqref{boundary} it follows that
$$
 F(\gamma^{*}v_2^+)=C-\frac{(v_2^+)^2}{2}.
$$
Differentiation of this expression with respect to $C$ yields
$$
( v_2^+)'= \frac{1}{v_2^+ +\gamma^{*}f(\gamma^{*}v_2^+)}.
$$
Hence
\begin{align} \label{diffLC}
\sqrt{\lambda}\frac{dL_{2}(C)}{dC}=\frac{2}{f(\gamma^{*}v_2^+)(v_2^++\gamma^{*}f(\gamma^{*}v_2^+))}-2\int_0^{v_2^+}\frac{f'(u(v))}{f^3(u(v))}dv.
\end{align}
 From \eqref{contact} we obtain
 \begin{align}\label{Deriv}
\lim_{C\searrow \tilde C}\frac{d}{dC} L_{2}(C)=-\infty.
\end{align}
\bigskip
Next we discuss the symmetric solutions represented in the phase plane by $\widehat{P^+_1P^-_1}$, i.e. $i=1$. They are defined for $C\in (\tilde C,\infty)$.
\medskip

\noindent
\begin{lemma}\label{Lsymm}
The length $L_{1}(C)$ satisfies
\begin{itemize} 
\item[(i)] $\lim_{C\to \infty} L_{1}(C)= -\frac{2}{\alpha}$,
\item[(ii)] $L_{1}(C)>-\frac{2}{\alpha} \tx{if} C\geq s_0$,
\item[(iii)] $L_{1}(\tilde C)=L_2(\tilde C))< -\frac{2}{\alpha}$,
\item[(iv)] $\lim_{C\to \tilde C}  \sqrt{\lambda}\frac{dL_{1}(C)}{dC}=\infty.$
\end{itemize}
\end{lemma}
\begin{proof}
(i)  If $C>s_0$,then
$\sqrt{2(C-s_0} <v_1^+< \sqrt{2C})$ and consequently $\gamma^{*}\sqrt{2C}<u_1^+<\gamma^{*}\sqrt{2(C-s_0}) $.
Introducing these estimates into  $
L_1(C)=2\int_{u_1^+}^{F^{-1}(C)} \frac{du}{\sqrt{\lambda}v(u)} $ (s. \eqref{LS}),
we obtain
$$
 2\frac{ F^{-1}(C)-\gamma^{*}\sqrt{2(C-s_0})}{\sqrt{\lambda2C}} \leq L_{1}(C)\leq 2\frac{F^{-1}(C)-\gamma^{*}\sqrt{2C}}{\sqrt{\lambda2(C-s_0)} }.
$$
The first claim follows by letting $C\to \infty$ and the superlinearity of $f$.
\medskip

\noindent
(ii) follows from \eqref{Greenb}
$$
u_{\max}=u\left(\frac{L}{2}\right)= \frac{\lambda}{\alpha}\int_0^{L/2}(1+\alpha \xi)f(u(\xi))\:d\xi.
$$
If f $C\geq s_0$, then $u_{\max}\geq 0$. Since $\alpha$ is negative, the integral has also to be negative. Hence $1+\alpha \frac{L}{2}<0$ which establishes (ii).
\medskip

\noindent
(iii) If $\tilde C<C<s_0$ it follows from the phase plane 
that 
\begin{align}\label{vergleich}
L_2(C)<L_1(C) \quad \tilde C<C<s_0 \quad\tx{and} \quad L_{2}(\tilde C)=L_{1}(\tilde C).
\end{align}
\medskip

\noindent
(iv) 
From \eqref{diffLC} we get (replacing  $L_2$ by $L_1$)
$$
\sqrt{\lambda}\frac{dL_1(C)}{dC}=\frac{2}{f(\gamma^{*}v_1^+)(v_1^++\gamma^{*}f(\gamma^{*}v_1^+))}-2\int_0^{v_1^+}\frac{f'(u(v))}{f^3(u(v))}dv.
$$
If $C\to \tilde C$, \eqref{contact} yields $\gamma^{*}f(\gamma^{*} v_1^+)+v_1^+=0$. 
Consequently
$
\lim_{C\to \tilde C}  \sqrt{\lambda}\frac{dL_1(C}{dC}=\infty.
$
\end{proof}
\begin{remark} \label{RE1} 
1. If $\alpha L=-2$, the Green's function doesn't exist.\\
2. The 1-d Steklov eigenvalue problem is of the form $\phi''=0$ in $(0,L)$, $-\phi'(0)=\mu \phi(0)$,
$\phi'(L)=\mu \phi(L)$. It has two eigenvalues $\mu_1=0$ and $\mu_2=\frac{2}{L}$. The eigenfunction corresponding to $\mu_2$ is $\phi(x)= c(x -\frac{L}{2})$.  If a solution of \eqref{general} exists for $\alpha =-2/L$, then it has to satisfy the compatibility condition
$$
\int_0^L((x -\frac{L}{2})f(u(x))\:dx=0.
$$
\end{remark}
Lemma \ref{Lsymm} leads to
\begin{lemma} \label{auxiliary} Let $\alpha$ and $\lambda$ be fixed. Consider $L_1(C)$ where $C\in (\tilde C,\infty)$. Let $\overline{L}_1:=\max_{C>\tilde{C}}L_1(C)$. Then 
\begin{itemize}
\item[(i)]
$L_1(C)$ attains its maximum at points $C_m \in (\tilde C,\infty)$. Hence for any $-\frac{2}{\alpha} <L<\overline{L}_1$ there exist at least two values of $C$ such that $L=L_1(C)$.
\item[(ii)] If $L=-\frac{2}{\alpha}$, there exist a bounded solution for some $C\leq s_0$. If $C\to \infty$ the corresponding  solutions become unbounded.
\end{itemize}
\end{lemma}
Next we consider problem \eqref{general}, \eqref{Robin} with fixed $L\in (L_1(\tilde C),\overline{L}_1)$.
 \begin{lemma} \label{2sol}
 Assume \eqref{conditions} and $f''>0$. For fixed $\alpha$ and $\lambda$ there exists at most two solutions.
 \end{lemma}
\begin{proof}
Suppose that there exist three solutions $u_i(x)$ $i=1,2,3$, corresponding in the phase to the trajectories $\widehat{P^+_1P^-_1}$ with $C_1<C_2<C_3$. This means that $\max u_1<\max u_2<\max u_3$. By Lemma \ref{alpha<0}
they intersect each other.  Suppose that $u_1(\gamma L)=u_2(\gamma L)=a$, where $0<\gamma<1/2$. In view of the symmetry $u_1(L(1-\gamma))=u_2(L(1-\gamma))=a$. 
Since $\max u_1<\max u_2<\max u_3$ and $u_3(0)=u_3(L)<u_i(0)$ for $i=1,2$ there exists an $\ell\in(0,\frac{L}{2})$ such that $u_3(\ell)=a$. By symmetry $u_3(L-\ell)=a$. There are three possible cases: $\ell=\gamma L$, $\ell\in(0,\gamma L)$ and $\ell\in(\gamma L,\frac{L}{2})$. The last two cases are illustrated by Figure \ref{fig:IntersectingSolutionsA} and \ref{fig:IntersectingSolutionsB}  

\begin{figure}[H]
	\centering
	\begin{subfigure}{0.4\textwidth}\centering
	\includegraphics[width=80mm]{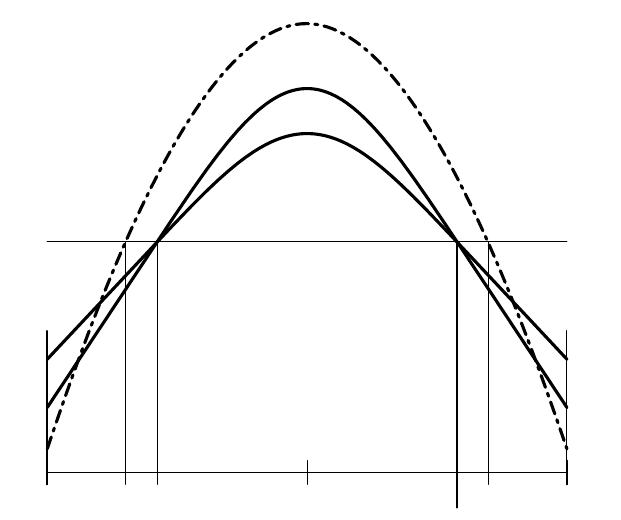}%
	\drawat{-44mm}{0.5mm}{$L/2$}%
	\drawat{-8mm}{0.5mm}{$L$}%
	\drawat{-75mm}{0.5mm}{$0$}%
	\drawat{-65mm}{0.5mm}{\small$\ell$}%
	\drawat{-61mm}{0.5mm}{\small$\gamma L$}%
	\drawat{-22.5mm}{-2mm}{\small$L\!-\!\gamma L$}%
	\drawat{-19.mm}{1.5mm}{\small$L\!-\!\ell$}%
	\drawat{-78mm}{35mm}{$a$}%
	\drawat{-34mm}{62mm}{$u_3$}%
	\drawat{-36.5mm}{55mm}{$u_2$}%
	\drawat{-38mm}{50mm}{$u_1$}%
	
	\caption{\label{fig:IntersectingSolutionsA}$0<\ell<\gamma L$}
	\end{subfigure}%
	\hspace{1cm}
	\begin{subfigure}{0.4\textwidth}\centering
	\includegraphics[width=80mm]{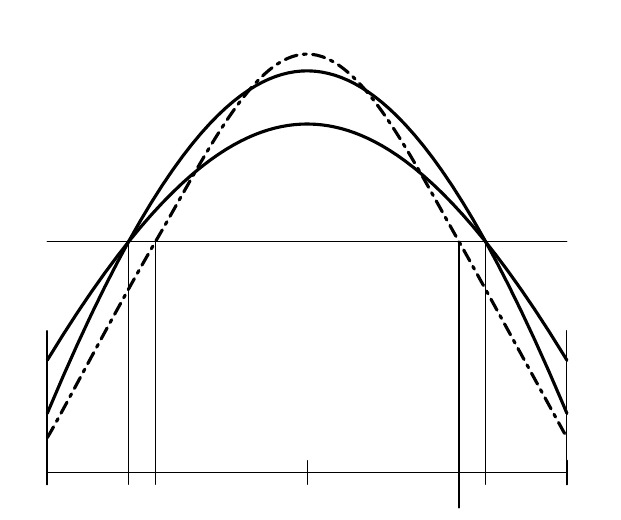}%
	\drawat{-44mm}{0.5mm}{$L/2$}%
	\drawat{-8mm}{0.5mm}{$L$}%
	\drawat{-75mm}{0.5mm}{$0$}%
	\drawat{-67.5mm}{0.5mm}{\small$\gamma L$}%
	\drawat{-60.7mm}{0.5mm}{\small$\ell$}%
	\drawat{-22.mm}{-2mm}{\small$L\!-\!\ell$}%
	\drawat{-19.5mm}{1.5mm}{\small$L\!-\!\gamma L$}%
	\drawat{-78mm}{35mm}{$a$}%
	\drawat{-36mm}{59mm}{$u_3$}%
	\drawat{-26.5mm}{50mm}{$u_2$}%
	\drawat{-38mm}{51.5mm}{$u_1$}%
	
	\caption{\label{fig:IntersectingSolutionsB}$\gamma L < \ell<\frac{L}{2}$}
	\end{subfigure}
	
	\caption{\label{fig:IntersectingSolutions}Intersecting Solutions}
\end{figure}

\medskip

\noindent
Set  $\tilde L= L(1-2\gamma)$ and consider the interval $(0,\tilde{L})$. After a suitable shift of the variable,
the functions $\tilde u_i(x):=u_i(x+ \gamma L)-a$ for $i=1,2$ are solutions of 
\begin{align} \label{DirichletH}
\tilde u_i''+ \lambda f(\tilde u_i +a)=0  \quad\tx{in} (0,\tilde L), \quad \tilde u_i(0)=\tilde u_i(\tilde L)=0.
\end{align}
Clearly $\tilde u_1(x)$ is the minimal solution in $(0,\tilde L)$. 
\medskip

\noindent
We distinguish between three cases.

1. If $\ell=\gamma L$, then $\tilde u_3(x)=u_3((x+\gamma L)-a$ is also a solution of \eqref{DirichletH}.
By Corollary \ref{C01} problem \eqref{DirichletH} has at most two solutions. Hence this situation is excluded.

2. Assume $\ell <\gamma L$ (s. Fig 8 (a)). Then $\tilde u_3(x)$ satisfies the same equation \eqref{DirichletH} as $\tilde u_i(x)$, $i=1,2$. However on the boundary $\tilde u_3(0)=\tilde u_3(\tilde L)>0$. Moreover $\tilde u_3(x)>\tilde u_2(x)$. Consider $d(x)= \tilde u_3-\tilde u_2$. Since $f''>0$ we obtain
$$
0=d''(x)+\lambda(f(\tilde u_3)-f(\tilde u_2))>d''(x)+\lambda f'(\tilde u_2)d(x).
$$
Barta's inequality implies
$$
\lambda >\nu_1,
$$
where $\nu_1$ is the lowest eigenvalue of $\phi'' + \nu f'(\tilde u_2 +a)\phi=0$ in $(0,\tilde L)$, $\phi(0)=\phi(\tilde L)=0$. Since $\tilde u_2$ is a non-minimal solution we can apply Lemma \ref{genprop1} 4. which holds also for Dirichlet boundary conditions and obtain a contradiction.

3. Let $\ell>\gamma L$. The function $\tilde u_3(x):= u_3(x-\ell)-a$ solves the same equation in $(0,\tilde \ell)$, $\tilde \ell= L-2\ell$ with $\tilde u_3(0)=\tilde u_3(\tilde \ell)=0$. The function $U(x)=\tilde u_3(\frac{\tilde L x}{\tilde \ell})$  is a solution of \eqref{DirichletH} with $\tilde \lambda:= \frac{\lambda \tilde L^2}{\tilde \ell^2}>\lambda$. It is well-known and follows also from the previous discussion
that the Dirichlet problem has for fixed $\tilde L$ at most two solutions, a minimal solution $v$ and a maximal solution $V$.

If $\lambda$ increases $\max v$ increases and $\max V$ decreases. Recall that $\tilde u_i$, $i=1,2$ are solutions corresponding  $\lambda <\lambda$. Since $\max U>\max \tilde u_2>\max \tilde u_1$ this is impossible.  Consequently
there is no solution $u_3(x)$. This completes the proof.
\end{proof}
{\sc Consequence.} For symmetric solutions the length $L(C)$ has for given $\lambda$ and $\alpha<0$ the following form.

\begin{figure}[H]
\centering
\includegraphics{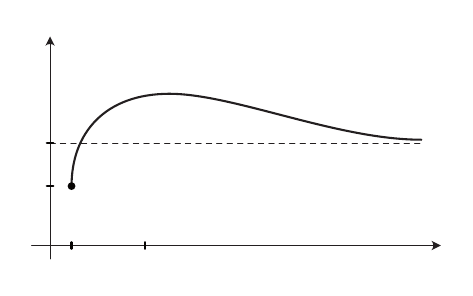}%
\drawat{-.8cm}{1cm}{$C$}%
\drawat{-5.7cm}{.35cm}{$s_0$}%
\drawat{-6.9cm}{.35cm}{$\tilde{C}$}%
\drawat{-8.4cm}{1.75cm}{$L_1(\tilde{C})$}%
\drawat{-8.4cm}{2.5cm}{$-2/\alpha$}%
\drawat{-7.7cm}{4.2cm}{$L_1$}%

\caption{$L_1(C)$}
\end{figure}
In summary we have 
\begin{theorem} Assume \eqref{conditions} and $f''>0$. Let $\alpha<0$ and $\lambda$ be fixed. Then 
\begin{itemize}
\item[(i)] for $-\frac{2}{\alpha}<L<\max L_1(C)$ there exist two symmetric solutions,
\item[(ii)] for $L=\max L_1(C))$ there exist one symmetric solution,
\item[(iii)] for $L<-\frac{2}{\alpha}$ there exist one symmetric solution.
\end{itemize}
\end{theorem}
\subsection{Asymmetric solutions}
Asymmetric solutions correspond to trajectories $\widehat{P^+_1P^-_2}$ and $\widehat{P^+_2P^-_1}$ in the phase plane. 

We restrict our discussion to $\widehat{P_1^+P_2^-}$ since the solution corresponding to $\widehat{P^+_2P^-_1}$ are obtained by reflexion at $L/2$.  

\noindent
The corresponding lengths of the interval are denoted by $L_{12}$ resp. $L_{21}$. Clearly $L_{12}=L_{21}$. As in the previous sections these lengths depend on $C$ for fixed $\lambda>0$ and $\alpha<0$.
\medskip

\noindent
According to Table 1 there are two types of asymmetric solutions. For $s_0\leq C$, the trajectory $\widehat{P_1^+P^-_2}$ corresponds to a monotone solution and for $\tilde C<C<s_0$ to a non-monotone solution.
\begin{lemma} \label{LA1}
\begin{enumerate}

\item For increasing asymmetric solutions ($C\geq s_0$)
$$
L_{12}(C)= \int_{v_2^-}^{v_1^+}\frac{dv}{\sqrt{\lambda}f(F^{-1}(C-\frac{v^2}{2}))}=\int_{u_1^+}^{u_2^-}\frac{du}{\sqrt{2\lambda(C-F(u))}}.
$$
Similarly the length of the decreasing asymmetric solutions is obtained by interchanging $ P^+_1$, $P_2^-$ by $P_2^+$, $P_1^-$. 
\item For non-monotone asymmetric solutions ($C\in (\tilde C,s_0)$)
\begin{align*}
L_{12}(C)=L_{21}(C)&=\int_0^{v_1^+}\frac{dv}{\sqrt{\lambda}f(F^{-1}(C-\frac{v^2}{2}))}+\int_{v_2^1}^0\frac{dv}{\sqrt{\lambda}f(F^{-1}(C-\frac{v^2}{2}))}\\
&=\int_{u_1^-}^{F^{-1}(C)}\frac{du}{\sqrt{2\lambda(C-F(u))}} + \int_{u_2^-}^{F^{-1}(C)}\frac{du}{\sqrt{2\lambda(C-F(u))}}.
\end{align*}
\end{enumerate}
\end{lemma}
\subsubsection{Monotone solutions}
We start with the discussion  monotone solutions. They are represented  by $\widehat{P^+_1P^-_2}$ and $\widehat{P^+_2P^-_1}$with  $C>s_0$. We restrict our discussion to the monotone increasing solutions $\widehat{P^+_1P^-_2}$ since the monotone decreasing solutions are obtained by reflexion are $L/2$.

\begin{figure}[H]
	\centering
	\begin{subfigure}{0.35\textwidth}\centering
	\includegraphics[width=60mm]{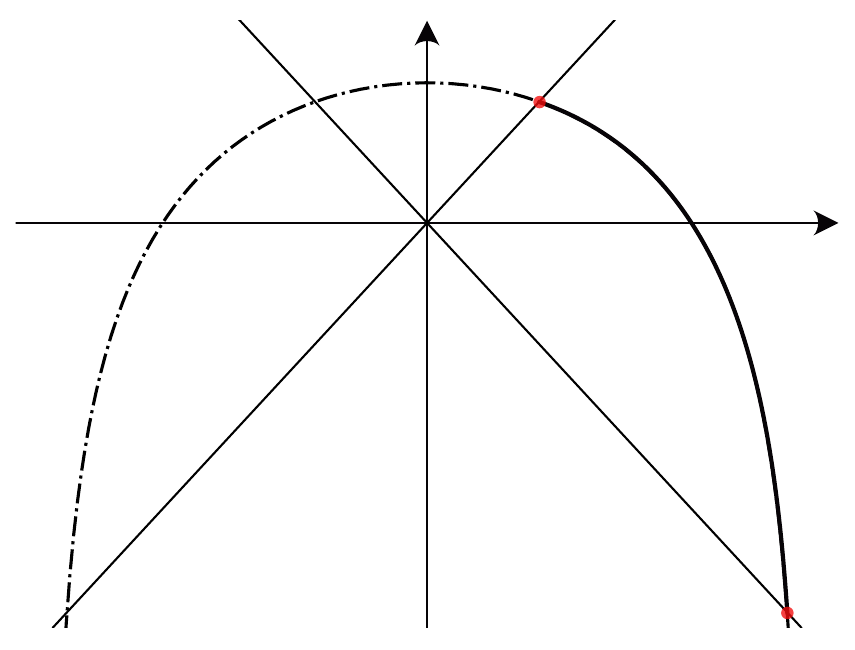}%
	\drawat{-4mm}{3.5mm}{$P^+_1$}%
	\drawat{-26mm}{41mm}{$P^-_2$}%
	\drawat{-33.5mm}{43mm}{$u$}%
	\drawat{-3mm}{31.5mm}{$v$}%

	\caption{Phase plane trajectory $\widehat{P^+_1P^-_2}$}
	\end{subfigure}%
	\begin{subfigure}{0.35\textwidth}
	\centering
	\includegraphics[width=50mm]{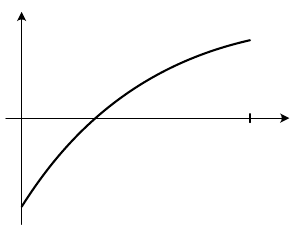}%
	\drawat{-3mm}{21.5mm}{$x$}%
	\drawat{-9.5mm}{15mm}{$L$}%
	
	\caption{Monotone solution}
	\end{subfigure}
	
	\caption{}
\end{figure}

The trajectory can be split in two pieces:
\begin{eqnarray*}
\widehat{P^+_1P^-_2}=\widehat{P^+_1P^{0}}\cup \widehat{P^{0}P^-_2}
\qquad\hbox{where}\qquad P^0=(\sqrt{2(C-s_0)},0).
\end{eqnarray*}
Hence the length $L_{12}$ consists of two pieces 
\begin{eqnarray}\label{LAM2}
L_{12}=L_{10}+L_{02}
\end{eqnarray} 
where
\begin{eqnarray*}
L_{10}(C)= \frac{1}{\sqrt{\lambda}}\int^{v_1^+}_{\sqrt{2(C-s_0)}}\frac{1}{f(u(v))}\:dv
\quad\hbox{and}\quad
L_{02}(C)= \frac{1}{\sqrt{\lambda}}\int^{\sqrt{2(C-s_0)}}_{v_2^-}\frac{1}{f(u(v))}\:dv.
\end{eqnarray*}
Clearly
\begin{align} \label{LM1}
\lim_{C\to s_0} L_{02}(C)=0.
\end{align}
\begin{lemma} \label{LemmaM1}
We have 
\begin{align*}
(i) \quad L_{02}(C)<-\frac{1}{\alpha},\\
(ii) \quad \lim_{C\to \infty}L_{02}(C)=0.
\end{align*}
\end{lemma}
\begin{proof}  The trajectory $\widehat{P^0P^-_2}$ corresponds to a solution of the boundary value problem
$u''(x) + \lambda f(u)=0$ in $(0,L)$ such that $u(0)=0$ and $u'(L)=-\alpha u(L)$. By means of the corresponding Green's function it can be written as an integral equation
$$
u(x)=\lambda \int_0^x \frac{1+\alpha L -\alpha x}{1 +\alpha L} \xi f(u(\xi))\:d\xi + \lambda \int_x^L\frac{1+\alpha L -\alpha \xi}{1+ \alpha L} xf(u(\xi))\:d\xi.
$$
Thus
$$
u_{\max}=u(L)= \frac{\lambda}{1+\alpha L} \int_0^L \xi f(u(\xi))\:d\xi.
$$
The first assertion follows from $u_{\max}>0$. Since $u(x)$ is concave we have $u(x)\geq \frac{u_{\max}}{L} x$. Hence
$$
u_{\max}\geq \frac{\lambda}{1+\alpha L} \int_0^L \xi f\left(\xi\frac{u_{\max}}{L}\right) \:d\xi.
$$
The change of variable $y=\frac{u_{\max}}{L}\,\xi$ and the convexity of $f$ lead to
\begin{align*}
u_{\max} \geq \frac{\lambda L}{(1+\alpha L)u_{\max}} \int_0^{u_{\max}}yf(y)\:dy&\geq\frac{\lambda L}{(1+\alpha L)u_{\max}} 
\int_0^{u_{\max}}y(f(0)+f'(0)y)\:dy\\
& =\frac{\lambda L}{(1+\alpha L)} \left( f(0)\frac{u_{\max}}{2}+ f'(0) \frac{u_{\max}^2}{3}\right).
\end{align*}
Hence $L\to 0$ as $u_{\max}\to \infty$ or equivalently $C\to \infty$.  
\end{proof}
Consider now the first term $L_{10}(C)$ in \eqref{LAM2}. The same type of arguments as for Lemma \ref{LemmaM1} imply
\begin{lemma} \label{LemmaM2} The length $L_{10}(C)$satisfies
\begin{align*}
(i) \quad L_{10}(C)>-\frac{1}{\alpha},\\
(ii) \quad \lim_{C\to \infty}L_{10}(C)= -\frac{1}{\alpha}.
\end{align*}
\end{lemma}
\begin{proof} 
The trajectory $\widehat{P^+_1P^0}$ corresponds to a solution of the boundary value problem
$u''(x) + \lambda f(u)=0$ in $(0,L)$ such that $u'0)=\alpha u(0)$ and $u(L)=0$. By means of the corresponding Green's function it can be written as an integral equation
$$
u(x)=-\lambda \int_0^x \frac{x-L}{1 +\alpha L} (\alpha \xi +1)f(u(\xi))\:d\xi - \lambda \int_x^L\frac{\alpha x+1}{1+\alpha L}(\xi-L)f(u(\xi))\:d\xi.
$$
Hence
$$
u(0)= \lambda \int_0^L \frac{L-\xi}{1+\alpha L} f(u(\xi))\:d\xi.
$$
The first statement follows from $u(0)<0$. The second statement is a consequence of $L_{10}(C)<L_1(C)/2$ and Lemma \ref{Lsymm}.
\end{proof}
In summary we have
\begin{theorem} \label{Thma1}Let $C>s_0$ and $\alpha<0, \lambda$ be fixed. Then 
$$
L_{12}(C)>-\frac{1}{\alpha}, \quad \lim_{C\to \infty} L_{12}(C)=-\frac{1}{\alpha}.
$$
If $L_{12}(C)$ is monotone decreasing in $C$, then for given $L\in (L_{12}(s_0), -\frac{1}{\alpha})$ there
exists one increasing monotone solution of \eqref{general} whereas if $L_{12}(C)$ isn't monotone, different solutions may appear.
\end{theorem}
The maximal number of monotone solutions depends on $L_1(C)$ where $C>s_0$.  
\begin{theorem} Under the assumptions of Lemma \ref{2sol} and if $L_1(C)$ is decreasing for $C>s_0$ there exist at most two increasing (decreasing) monotone solutions.
\end{theorem}
\begin{proof} Suppose that Problem \eqref{general} with Robin boundary conditions has three solutions $u_i(x)$, $i=1,2,3.$
corresponding to the trajectories $s_0<C_3<C_2<C_1$. Complete $u_i(x)$ to a symmetric solution $U_i(x)$ in the interval $(0,L_1(C_i))$. By our assumption $L_1(C_1)>L_1(C_2)>L_1(C_3)$. Suppose that all $U_i$ - after a possible shift in $x$ - attain their maximum at the origin. Since
$U_1(-L_1(C_1)/2)<U_2(-L_1(C_2)/2)<U_3(-L_1(C_3)/2)$ and $U_1(L_1(C_1)/2)>U_2(L_1(C_2)/2)>U_3(L_1(C_3)/2)$ 
we are in the same situation as in Figure \ref{fig:IntersectingSolutions}. The remainder of the proof is now the same as the one of Lemma \ref{2sol}.
\end{proof}

{\sc Example} If $f(u)=e^u$ then
$$
L_{12}(C)= \sqrt{\frac{2}{C\,\lambda}}\left(\arctanh\left(\sqrt{1-\frac{e^{u^+_1}}{C}}\right)
- \arctanh\left(\sqrt{1-\frac{e^{u^-_2}}{C}}\right)\right).
$$
Here $u_1^+$ and $u_2^-$ are the negative resp. positive roots of
\begin{equation}\label{eq:u1pu2nGelfandLength}
2C-\left(\frac{u}{\gamma^*}\right)^2= 2 e^u.	
\end{equation}
{\color{gray}

\begin{figure}[H]
\centering
\begin{subfigure}[t]{8.1cm}
\centering
\includegraphics[height=6cm]{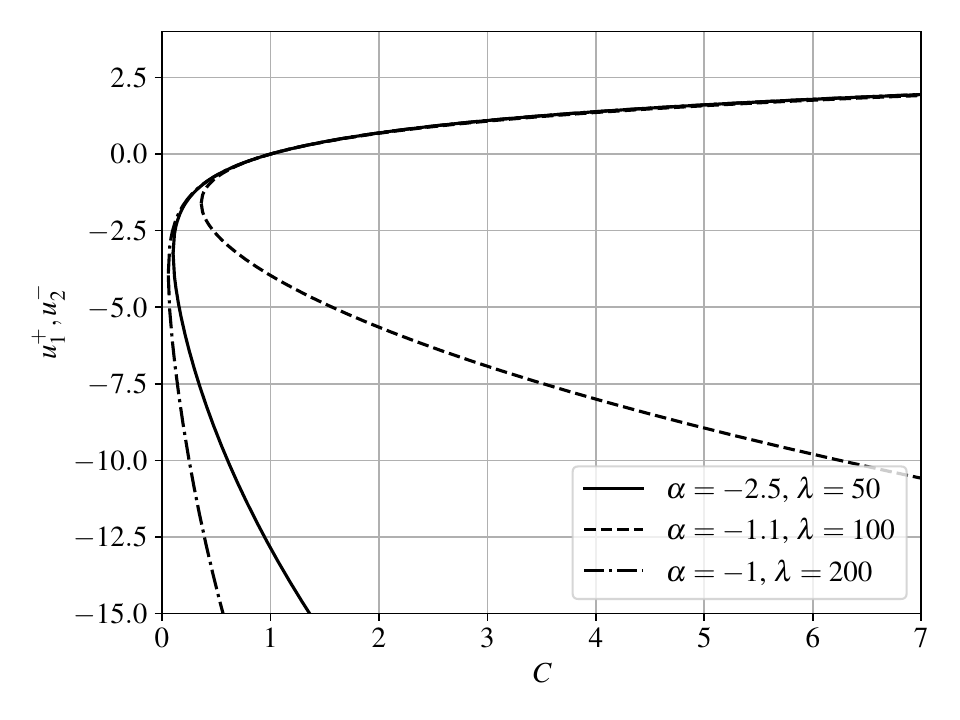}%
\caption{Negative resp. positive roots of (\ref{eq:u1pu2nGelfandLength})}
\end{subfigure}
\begin{subfigure}[t]{8.1cm}
\centering
\includegraphics[height=6cm]{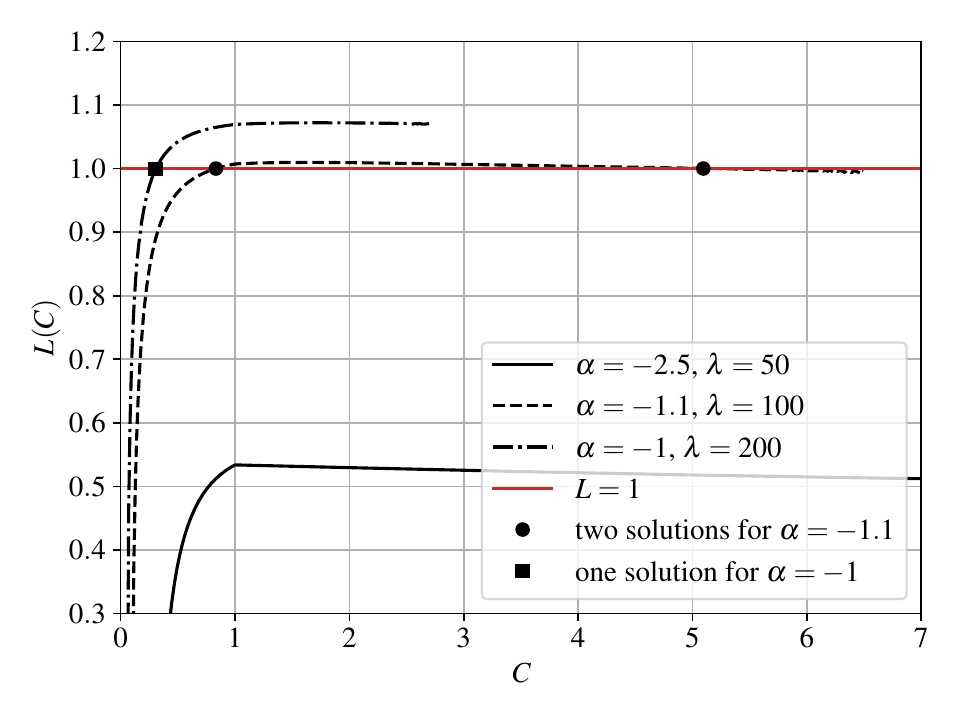}%
\caption{Length $L_{12}(C)$ for selected $\alpha$'s and $\lambda$'s.}
\end{subfigure}
\caption{\label{fig:ExampleLC12Gelfand}$L_{12}(C)$ for $f(u) = e^u$ and the number of solutions for a given length $L=1$.}
\end{figure}
\medskip
}
\subsubsection{Non - monotone solutions}
Consider now non-monotone asymmetric solutions. They correspond to the trajectories  $\widehat{P^+_1P^-_2}$ and $\widehat{P^+_2P^-_1}$ and to $C\in (\tilde C,s_0)$ (see Figure \ref{fig:NonMonotoneSolutions}). 

\begin{figure}[H]
	\centering
	\begin{subfigure}{0.35\textwidth}\centering
	\includegraphics[width=60mm]{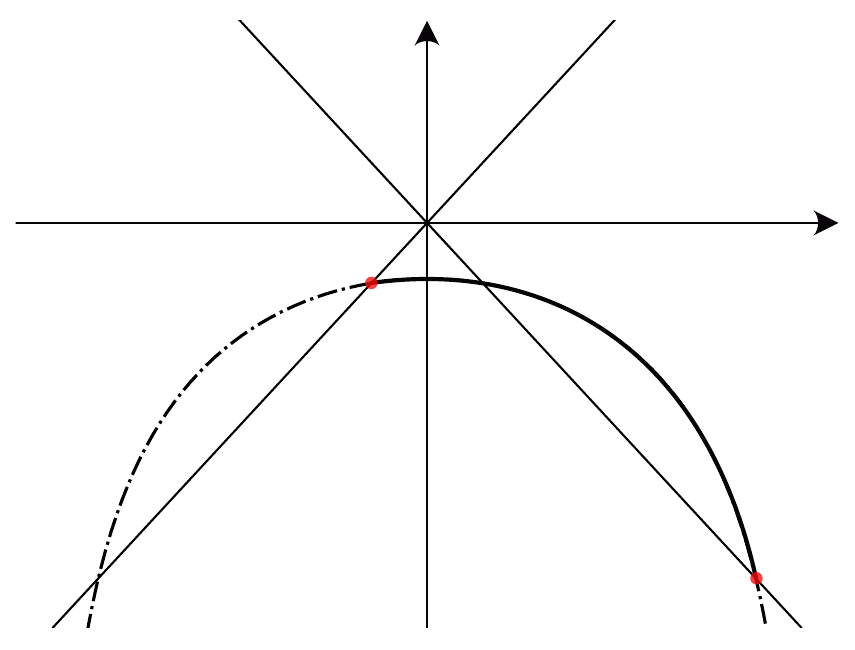}%
	\drawat{-6mm}{5mm}{$P^+_1$}%
	\drawat{-40mm}{26.5mm}{$P^-_2$}%
	\drawat{-33.5mm}{43mm}{$u$}%
	\drawat{-3mm}{31.5mm}{$v$}%

	\caption{Phase plane trajectory $\widehat{P^+_1P^-_2}$}
	\end{subfigure}%
	\begin{subfigure}{0.35\textwidth}
	\centering
	\includegraphics[width=50mm]{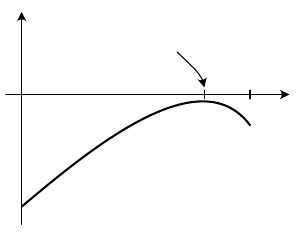}%
	\drawat{-3mm}{25.5mm}{\small$x$}%
	\drawat{-9.mm}{25.5mm}{\small$L$}%
	\drawat{-30.mm}{33mm}{\small$L-L_2(C_1)/2$}%
	
	\caption{Non-monotone solution}
	\end{subfigure}
	
	\caption{\label{fig:NonMonotoneSolutions}}
\end{figure}

Obviously also in this case we have $L_{12}=L_{21}$. 
Clearly
$$
2L_{12}(C)= L_1(C)+L_2(C).
$$
Next we discuss the number of  solutions which are represented by the trajectories  $\widehat{P_1^+P_2^-}$. 

\begin{lemma}\label{2solb} Assume \eqref{conditions} and $f''>0$. Then Problem \eqref{general} with Robin boundary conditions has at most  two asymmetric non-monotone solutions corresponding to the trajectories $\widehat{P_1^+P_2^-}$. 
\end{lemma}
\begin{proof} 
Let $u_i(x)$ be three solutions depending on $C_i$ for $i=1,2,3$. We will show that as in Lemma \ref{2sol} this leads to a contradiction. The next observations are based on the phase plane.
\medskip

\noindent
1. For a given $C_1<s_0$ the trajectory $\widehat{P^+_1P^-_2}$ (see Figure \ref{fig:NonMonotoneSolutions}) 
corresponds to a solution $u_1$ of \eqref{general} and \eqref{Robin} with 
$u_1(0)<u_1(L)$. It is increasing in the interval $(0,L-\frac{L_2(C_1)}{2})$ and then 
decreasing in $(L-\frac{L_2(C_1)}{2},L)$. It is symmetric in $(L-L_2(C_1),L)$ with respect to reflections in the point $L-\frac{L_2(C_1)}{2}$. For $C_2<C_1<s_0$ there is a solution $u_2$ with the same properties, when we replace $C_1$ by $C_2$. Now assume there is a solution $u_3$ for some constant $C_3<C_2<C_1<s_0$. Thus in $(0,L)$ we assume there exist three solutions $u_i$, $i=1,2,3$ corresponding to 
$\tilde C<C_3<C_2<C_1<s_0$.
\medskip

\noindent
2. By the phase plane $u_1(0)<u_2(0)<u_3(0)$ and $u_1(L)>u_2(L)>u_3(L)$. Thus all solutions must intersect each other at least once. Let us consider $u_1$ and $u_2$. Since $L_2(C)$ is monotone decreasing and $u_2(x)<u_1(x)$ in $L-L_2(C_1)$ , the function $u_2$ intersects $u_1$ at some $0<\ell_{12}<L-L_2(C_1)$. Similarly $u_3(x)$ intersects $u_2(x)$ in $\ell_{32}<L-L_2(C_2)$ and $u_1(x)$,  $\ell_{31}>L-L_2(C_1)$. 
\medskip

\noindent
3. Since $L_2(C)$ is monotone decreasing and $u_2(x)<u_1(x)$ in $L-L_2(C_1)$ , $u_2(x)$ intersects $u_1(x)$ at some $\ell_{12}<L-L_2(C_1)$. Note that $u_1$ is monotone on $(0,L-L_2(C_1)$. A concave function ($u_2$) intersects a monotone function ($u_1$) at most once. Thus $\ell_{12}$ is the unique.
Similarly $u_3(x)$ intersects $u_2(x)$ only in some $\ell_{32}<L-L_2(C_2)$ and $u_1$ intersects  $\ell_{31}>L-L_2(C_1)$. 
\medskip

\noindent 
4. Since $u_{\max}= F^{-1}(C)$ we have $\max u_1>\max u_2>\max u_3$. Complete now $u_i(x)$ to a symmetric solution $U_i(x)$ in the interval $(0, L_1(C_i))$. Clearly $L_1(C_1)>L_1(C_2)>L_1(C_3)$. Let us shift all $U_i(x)$ such that they attain their maximum in the origin. This leads to the same configuration  as in Lemma \ref{2sol}. The  remainder of the proof is as for Lemma \ref{2sol}.
\end{proof}

\section{Numerical Results}
In order to compute for fixed $L, \alpha$ and $\lambda$ the number of solutions of problem \eqref{general} we write it in the weak form
\begin{equation}\label{eq:Gulambda}
G_\alpha(u,\lambda) := \int_0^L \big(u'\cdot v' - \lambda f(u) v \big) dx + \alpha (u(L) v(L) + u(0) v(0))  = 0 \quad \forall v\in V_h.
\end{equation}
We use a high order finite element discretization $V_h$ \cite{NGSolve}  to compute the path 
\begin{align*}
\Gamma_{\alpha} := \{(u(s), \lambda(s))\ |\ u(s)\in V_h,\ G_\alpha(u(s),\lambda(s)) = 0\ \forall s\in I\subset \R\},
\end{align*}
given by the functions $u$, the associated parameter $\lambda$ and a pseudo-arc length $s$ .  At turning points on a regular path, Newton's method fails. More sophisticated methods are needed. We use the pseudo-arc length continuation (cf. for instance \cite{Mi}). All solution paths are computed for L = 1 fixed.

We obtain symmetric solution paths 
if we start at $\lambda = 0$. To be able to calculate the path of asymmetric solutions, we need a solution for a given $\lambda$. The path can be continued from this solution. In the case of the Bratu-Gelfand problem, we can calculate this solution analytically. If the solution is not given analytically, we calculate solutions for a fixed $\lambda$ using shooting methods and project them into the finite element space to calculate the complete path. In addition to the Bratu-Gelfand problem, numerical results for two further problems are shown later.
\medskip

\noindent
{\bf Bratu-Gelfand problem.}
The Figure \ref{fig:PathesGelfandEquationbw} shows solution paths for the Bratu-Gelfand equation $u''+ \lambda e^u = 0$ for different $\alpha<0$. For a fixed $\lambda$, the number of solutions can be determined from the diagram. Recall that the asymmetric solutions must be counted twice by reflection. For example, we obtain a total of five solutions for $\lambda = 100$ and $\alpha = -1.1$. On the other hand, for $\lambda = 250$ and $\alpha = -1$  only three solutions exists.

\begin{figure}[h]
\centering
\includegraphics[width=15cm]{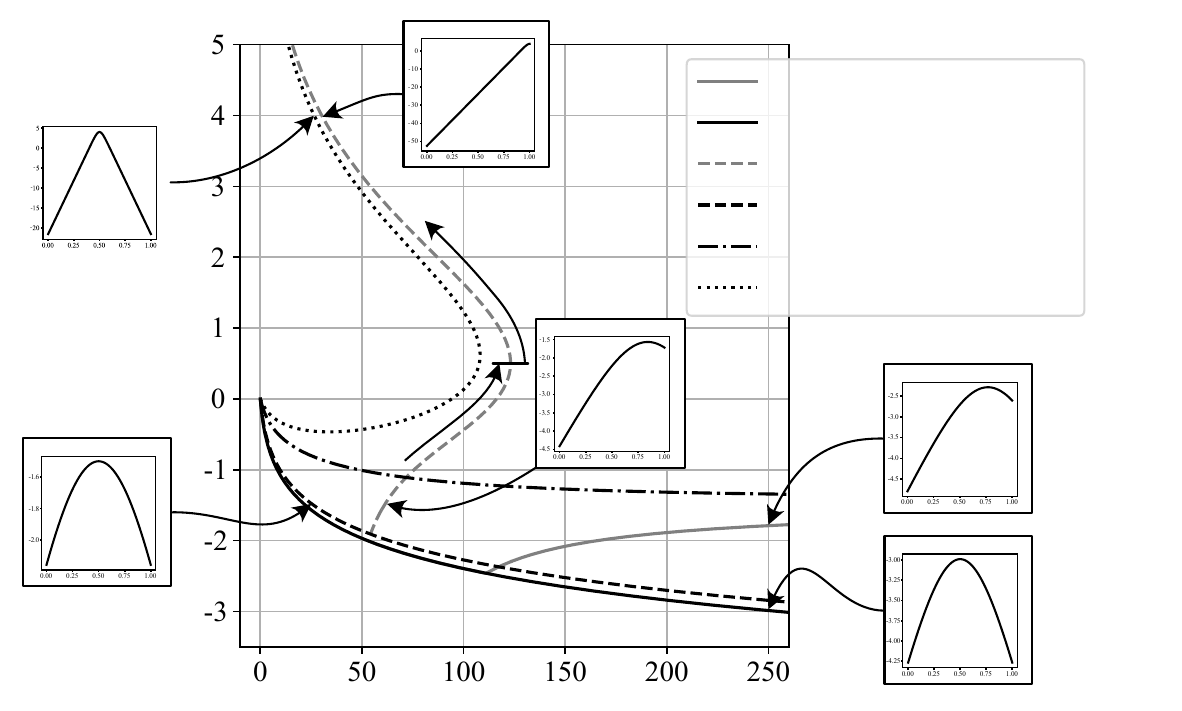}%
\drawat{-5.15cm}{7.9cm}{\small asymmetric, $\alpha=-1$}%
\drawat{-5.15cm}{7.4cm}{\small symmetric, $\alpha=-1$}%
\drawat{-5.15cm}{6.85cm}{\small asymmetric, $\alpha=-1.1$}%
\drawat{-5.15cm}{6.35cm}{\small symmetric, $\alpha=-1.1$}%
\drawat{-5.15cm}{5.8cm}{\small symmetric, $\alpha=-2$}%
\drawat{-5.15cm}{5.3cm}{\small symmetric, $\alpha=-2.5$}%
\drawat{-8.5cm}{0cm}{\small$\lambda$}%
\drawat{-9.4cm}{6.2cm}{\scriptsize\textcircled{\drawat{-.75mm}{-.3mm}{\scriptsize{2}}}}%
\drawat{-8.6cm}{7.25cm}{\scriptsize\textcircled{\drawat{-.75mm}{-.3mm}{\scriptsize{2}}}}%
\drawat{-10.cm}{3.4cm}{\scriptsize\textcircled{\drawat{-.75mm}{-.3mm}{\scriptsize{1}}}}%
\drawat{-6.9cm}{3.45cm}{\scriptsize\textcircled{\drawat{-.75mm}{-.3mm}{\scriptsize{1}}}}%
\drawat{-13.cm}{3.7cm}{\rotatebox{90}{\small$\displaystyle\max_{x\in[0,1]} u(x)$}}%

\parbox{15cm}{\caption{\label{fig:PathesGelfandEquationbw}Solution pathes and selected solutions for the Gelfand equation. On the path for asymmetric solutions for $\alpha=-1.1$ we have non-monotone asymmetric solutions up to the turning point {\scriptsize\textcircled{\drawat{-.75mm}{-.3mm}{\scriptsize{1}}}} and after monotone symmetric solutions {\scriptsize\textcircled{\drawat{-.75mm}{-.3mm}{\scriptsize{2}}}}}}
\end{figure}

\noindent
{\bf Other convex functions.} The first alternative problem is given by
\begin{equation}\label{eq:alternativeI}-u''(x) = \lambda \begin{cases}\frac{1}{(u(x)-1)^2} & \quad u(x) < 0\\
1 + 2\, u + 3\, u^2 & \quad\text{else}.\end{cases}
\end{equation}
The nonlinearity satisfies the necessary conditions (\ref{conditions}). Analogous to the Bratu-Gelfand equation (see Figure \ref{fig:PathesGelfandEquationbw}), we obtain symmetric as well as asymmetric solutions. The solution paths are shown in Figure \ref{fig:PathesAlternativeI2bwMod}.

\begin{figure}[H]
\centering
\includegraphics[width=13.5cm]{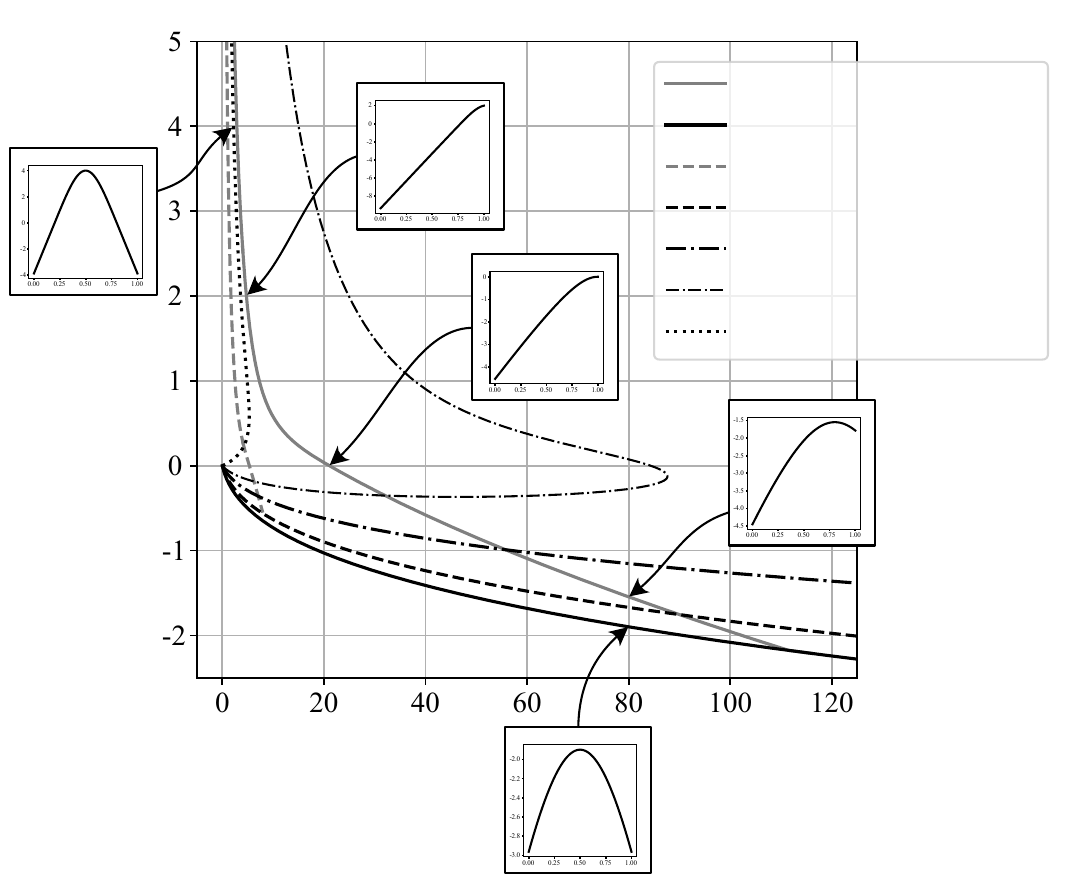}%
\drawat{-4.1cm}{10.1cm}{\small asymmetric, $\alpha=-1.3$}%
\drawat{-4.1cm}{9.6cm}{\small symmetric, $\alpha=-1.3$}%
\drawat{-4.1cm}{9.05cm}{\small asymmetric, $\alpha=-1.5$}%
\drawat{-4.1cm}{8.5cm}{\small symmetric, $\alpha=-1.5$}%
\drawat{-4.1cm}{8.0cm}{\small symmetric, $\alpha=-2$}%
\drawat{-4.1cm}{7.5cm}{\small symmetric, $\alpha=-2.6$}%
\drawat{-4.1cm}{7cm}{\small symmetric, $\alpha=-5$}%
\drawat{-3.6cm}{1.9cm}{\small$\lambda$}%
\drawat{-12.1cm}{5.5cm}{\rotatebox{90}{\small$\displaystyle\max_{x\in[0,1]} u(x)$}}%

\parbox{14cm}{\caption{\label{fig:PathesAlternativeI2bwMod}Solution paths and selected solutions for the alternative equation (\ref{eq:alternativeI}).}}
\end{figure}

A second alternative problem only for $\alpha > 0$ is given by
\begin{equation}\label{eq:alternativeII}-u''(x) = \lambda (1+u(x)^2).\end{equation}
Since $f(u) = 1+u^2$ it satisfies \eqref{conditions} only for $u\ge 0$. Let $\alpha > 0$. In this case there are only symmetrical solutions. Figure \ref{fig:PathesAlternativeEqIIbw} shows solution paths for selected alpha's.

\begin{figure}[H]
\centering
\includegraphics[width=8cm]{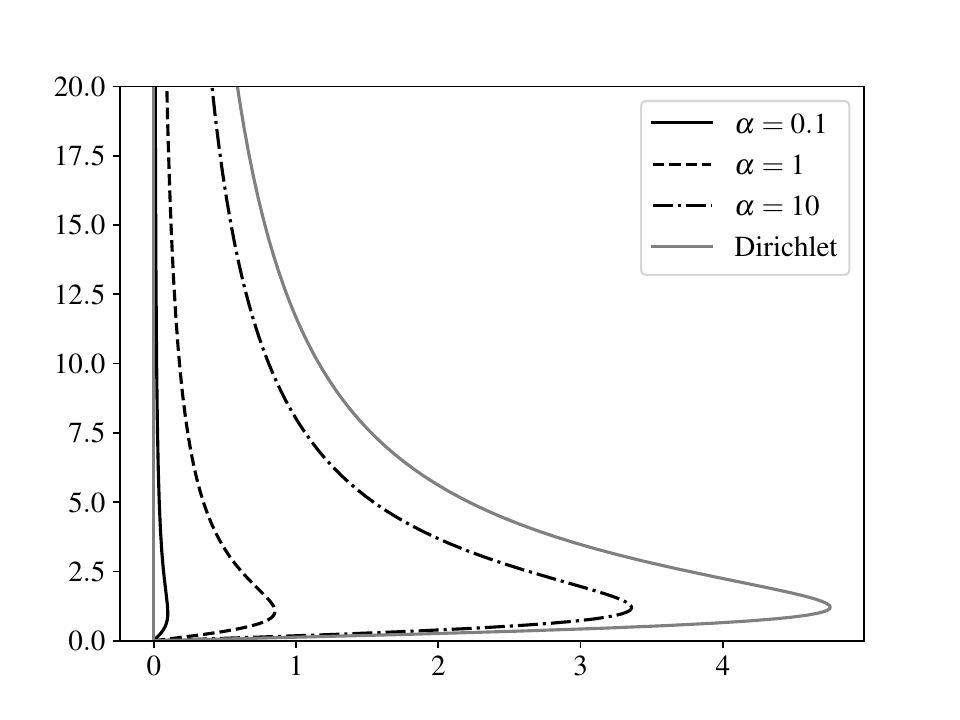}%
\drawat{-3.8cm}{0.1cm}{\small$\lambda$}%
\drawat{-8.3cm}{2.5cm}{\rotatebox{90}{\small$\displaystyle\max_{x\in[0,1]} u(x)$}}%

\caption{\label{fig:PathesAlternativeEqIIbw}Solution paths for (\ref{eq:alternativeII}).}
\end{figure}

\noindent
{\bf Non-convex function.} We consider the problem
\begin{equation}\label{eq:nonConvexExample}
	\begin{split}
	-u''(x) & = \lambda \left(\frac{256 u^5}{45}-\frac{64 u^4}{3}+\frac{64 u^3}{3}+u+1\right)\\
	u(0) & = u(1) = 0.
\end{split}
\end{equation}
Figure \ref{fig:convexExampleP5SolutionPath} indicates the existence of four solutions for special values of $\lambda$.

\begin{figure}[H]
	\centering
	\begin{subfigure}[t]{0.4\textwidth}\centering
	\includegraphics[width=60mm]{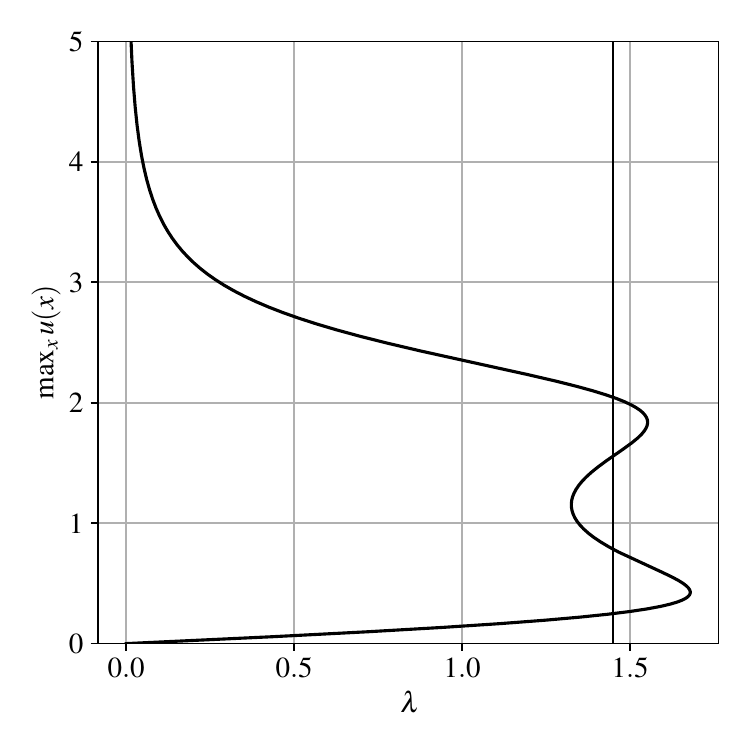}%
	
	\parbox{60mm}{\caption{Solution path for Dirichlet boundary conditions.}}
	\end{subfigure}
	\begin{subfigure}[t]{0.4\textwidth}\centering
	\includegraphics[width=60mm]{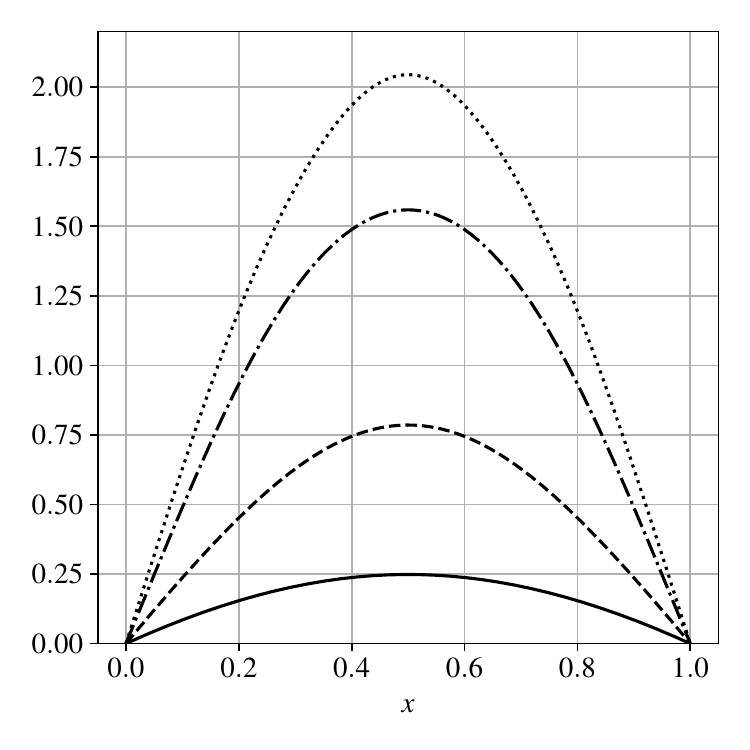}%
	
	\parbox{60mm}{\caption{Solutions for $\lambda = 1.45$.}}
	\end{subfigure}

	\caption{\label{fig:convexExampleP5SolutionPath}Solution path for the non-convex problem (\ref{eq:nonConvexExample}).}
\end{figure}


\end{document}